\documentclass[12pt]{amsart}
\usepackage[cp1251]{inputenc} 
\usepackage[english]{babel}
\usepackage{amsthm,amssymb,amsmath,amsfonts,mathrsfs, mathtools}
\usepackage{amsmath}
\usepackage{fixltx2e}
\usepackage{geometry}
    \usepackage[colorlinks=false]{hyper ref}
\DeclareRobustCommand{\divby}{%
  \mathrel{\text{\vbox{\baselineskip.65ex\lineskiplimit0pt\hbox{.}\hbox{.}\hbox{.}}}}%
}

\usepackage{pst-node}
\usepackage{tikz-cd} 
\usepackage{makecell}

\usepackage{mathtools, mathrsfs}
\makeatletter
\newcommand*{\rom}[1]{\expandafter\@slowromancap\romannumeral #1@}
\makeatother

\makeatletter
\@addtoreset{equation}{section}
\makeatother

\newtheorem{theorem}[equation]{Theorem}
\newtheorem{proposition}[equation]{Proposition}
\newtheorem{lemma}[equation]{Lemma}
\newtheorem{corollary}[equation]{Corollary}

\theoremstyle{definition}

\newtheorem{definition}[equation]{Definition}
\theoremstyle{remark}
\newtheorem{remark}[equation]{Remark}

\usepackage{hyperref}

\def \Hom {\operatorname{Hom}}
\def \End {\operatorname{End}}

\def \ind {\operatorname{ind}}

\def \GL {\operatorname{GL}}

\def \rk {\operatorname{rk}}

\newcommand{\CC}{\mathbb{C}}
\newcommand{\ZZ}{\mathbb{Z}}
\newcommand{\NN}{\mathbb{N}}

\author{Iuliya Beloshapka}
\address{ETH Zurich, Switzerland}
\email{iuliya.beloshapka@math.ethz.ch}

\title{Irreducible representations of the group of unipotent matrices of order $\mathbf{4}$ over integers}

\begin{document}

\maketitle

\begin{abstract}
We study a coarse moduli space of irreducible representations of the group of unipotent matrices of order $4$ over the ring of integers which have finite weight. All such representations are known to be monomial (see~\cite{bel-gor}). To describe a coarse moduli space of such representations, we need to study pairs of subgroups and their characters, which induce non-isomorphic irreducible representations. We obtain a full classification of such pairs and, respectively, a coarse moduli space.
\end{abstract}

\section{Introduction}

Moduli spaces of irreducible representations of finitely generated nilpotent groups are supposed to be used in questions related to $L$-functions of varieties over finite fields (see~\cite{ParshinCong}). It does not seem reasonable to study such moduli spaces in full generality, so one should restrict the class of all irreducible objects. Brown in~\cite{Brown} introduced the notion of a finite weight representation. A representation $\pi$ of a group~$G$ has finite weight if there is a subgroup $H\subset G$ and a character $\chi$ of $H$ such that the vector space $\Hom_H(\chi,\pi\vert_H)$ is non-zero and finite-dimensional. A representation $\pi$ is called monomial, if there exist a subgroup $H \subset G$ and a character $\chi:~H~\rightarrow~\CC^{*}$ such that $\pi~\simeq~\ind_{H}^{G}(\chi)$. In the plenary lecture at ICM $2010$, Parshin conjectured that irreducible representations of finitely generated nilpotent groups are monomial if and only if they have finite weight (see~\cite[\S~5.4(i)]{ParshinCong} for details). The conjecture was proven in a joint work with Gorchinskiy~\cite{bel-gor}. This allows us to approach the moduli problem of irreducible representations $\pi$ which have finite weight, since they always correspond to certain pairs $(H,~\chi)$ such that $\pi~\simeq~\ind_{H}^{G} (\chi)$.

Parshin and Arnal have studied in detail the case of the Heisenberg group over the integers~\cite{Parshin-Arnal}. For this group, they constructed a parameter space (i.e., a coarse moduli space) of complex irreducible representations which have finite weight. It turns out that the parameter space consists of two parts, corresponding to finite-dimensional irreducible representations and infinite-dimensional ones. The first one is a countable disjoint union of copies of $ \CC^{*}  \times \CC^{*} $. The second one, in turn, is a countable disjoint union of elliptic fibrations over $\CC^{*} \setminus S^1$ and components which do not have a complex variety structure (see~\cite{ParshinCong},~\cite{Parshin} for details). The question was also studied in the case of mixed real and integer coefficients, which was motivated by the theory of two-dimensional local fields~\cite{bel-1}.

The construction looks similar to Kirillov's orbit method~\cite{orbit} for connected real or complex nilpotent Lie groups. Attempts to extend Kirillov's method to $p$-adic nilpotent groups were made in~\cite{Boyarchenko}. Also, there exists an analogue of Kirillov's character formula for the discrete Heisenberg group.

Our main result is the construction of a coarse moduli space of irreducible representations of the group of unipotent $4 \times 4$ matrices over integers which have finite weight (see Theorem~\ref{theorem:main} and Table~\ref{table}). We denote this group by~$G$. In other words, we provide a full classification of pairs $(H, \chi)$ such that corresponding representations $\ind_{H}^{G}(\chi)$ are irreducible.

\bigskip

\begin{itemize}
\item[]{}
{\it {\sc Theorem.} There is a one-to-one correspondence between the following spaces:

$1.$ The union of the total spaces of the following bundles: $X_{1, 1} \rightarrow \Xi_{1,1}$, $X_{2, 0} \rightarrow \Xi_{2,0}$, $X_{2, 1} \rightarrow \Xi_{2,1}$, $X_{1, 2} \rightarrow \Xi_{1,2}$, $X_{2, 2} \rightarrow \Xi_{2,2}$, and $X_{3, 2} \rightarrow \Xi_{3,2}$. 

$2.$ A coarse moduli space of irreducible representations for the group of unipotent matrices of order $4$ with integer entries which have finite weight.

\medskip

A map from $X_{r_1, r_2} \rightarrow \Xi_{r_1,r_2}$ to the set of irreducible monomial representations is defined as follows: 
$$(H, \chi) \longmapsto \ind_{H}^{G}(\chi).$$\
}
\end{itemize}

\bigskip

First steps towards the moduli space problem were made in~\cite{bel-2}. The coarse moduli space has a natural iterated structure of a bundle over the set of certain subgroups $H \subset G$ (see Theorem~\ref{theorem:main}). It turns out that the number of isomorphism classes of irreducible finite weight representations of discrete nilpotent groups increases very rapidly, while a nilpotency class increments only by one. Namely, for the Heisenberg group over the integers, there are only two substantially different cases of weight pairs $(H, \chi)$ which correspond to irreducible monomial representations. In turn, there are over $50$ different cases for the group $G$, which makes our classification quite technical and lengthy. The developed techniques may be used for further generalizations to finitely generated nilpotent groups of higher nilpotency classes. 

\bigskip

The paper is organized as follows. In Section $2$, we provide results which concern an arbitrary finitely generated nilpotent group, some of which are well known. In Subsection $2.1$ we collect well-known formulas for endomorphisms of finitely induced representations, based on Frobenius reciprocity and Mackey's formula. Also, we introduce a notion of an irreducible weight pair. In Subsection $2.2$ we recall a result from~\cite{bel-gor} which relates ireducibility and Schur irreducibility of finitely induced representations. In Subsection $2.3$ we prove a result which concerns irreducible weight pairs with abelian subgroups of finitely generated torsion-free nilpotent groups. In Subsection $2.4$ we recall a criterion from~\cite{bel-gor} for irreducible weight pairs to be equivalent. We also prove there a result which concerns equivalent weight pairs with subgroups which have the same radical subgroup. In Subsection $2.5$ we recall a definition of ranks of finitely generated nilpotent groups.

All the other sections concern the group of unipotent matrices of order $4$ over the ring of integers. 

In Section $3$ we list the ranks of subgroups which may appear for this group and which provide irreducible monomial representations. The next six sections are organised in a similar way as we are dealing with six possible cases of ranks of subgroups. In each of those sections first we describe generators of subgroups of given ranks, and then we obtain the conditions on a character of a subgroup that the weight pair is irreducible. After that, we study equivalent irreducible weight pairs such that a subgroup is the same, but characters are different. Then we find all irreducible weight pairs which are equivalent to a given one such that subgroups in those pairs are different.

Section $10$ sums up all the possible cases of ranks of subgroups in a classification theorem, which is the main result of the paper.

\medskip

\section{Preliminaries}

We recall some well-known facts. Let $G$ be an arbitrary group, and $H$ be a subgroup of a group $G$. We use notation from~\cite{bel-gor}.

\subsection{Endomorphisms of finitely induced representations.}
\begin{definition} \label{def:ss}
Let $S(H) \subset G$ be the set of all elements $g \in G$ such that the index of $H^g \cap H$ in $H$ is finite.
\end{definition}

Let $\chi: H \rightarrow \CC^{*}$ be a character of a subgroup $H$.

\begin{proposition}\label{prop:end}
There is a canonical isomorphism of vector spaces
$$
\End_{G} \big( \ind^{G}_{H}(\chi) \big) \simeq \bigoplus_{ \bar g \in H \backslash S(H) / H }  \Hom_{H^g \cap H} \big( \chi \vert_{H^g \cap H},  \ \chi^{g} \vert_{H^{g} \cap H}  \big) \, .
$$
\end{proposition}

Proposition~\ref{prop:end} motivates the following definition.

\begin{definition}\label{def:perf}
Let $S(H, \chi) \subset G$ be the set of all elements $g \in S(H)$ such that
$$
\Hom_{H^g\cap H} \big(\chi \vert_{H^g \cap H}, \chi^{g} \vert_{H^g \cap H}\big)\ne 0\,,
$$

\begin{center}
or, equivalently,
\end{center}
$$\chi \vert_{H^g \cap H} = \chi^{g} \vert_{H^g \cap H}.$$
\end{definition}

As an immediate corollary from the Proposition~\ref{prop:end}, we obtain a canonical isomorphism of vector spaces
$$
\End_{G} \big( \ind_{H}^{G} (\chi) \big) \simeq \bigoplus_{H \backslash S(H, \chi) / H} \CC \, .
$$

\begin{definition}\label{definition:irr}
\hspace{0cm}
An {\it irreducible weight pair} is a pair $(H,\chi)$ such that $H \subset G$ is a subgroup,~$\chi$ is a character of $H$, and~$\ind_{H}^{G} (\chi)$ is an irreducible representation.
\end{definition}

\begin{remark}
Unfortunately, this definition is in a way an abuse of notation. In the paper~\cite{bel-gor} an {\it irreducible pair} was defined as a pair~$(H,\rho)$, where $H\subset G$ is a subgroup and~$\rho$ is a (non-zero) finite-dimensional irreducible representation of~$H$. 
\end{remark}

\subsection{Irreducibility vs. Schur irreducibility.}

We recall that a representation $\pi$ is called { \it Schur-irreducible} if $\End_{G}(\pi) = \CC$. 

We also recall that a generalization of Schur's lemma to countable groups holds true. Namely, any countably dimensional irreducible representation over $\CC$ of an arbitrary group is Schur-irreducible (see, e.g.,~\cite[Claim 2.11]{Bernstein}).

\begin{remark}\label{rem-cent}
For an irreducible weight pair $(H, \chi)$, the centralizer $C_{G}(H) = \{ g \in G \ \vert \ [g,h]~=~\mathrm{1} \ \text{for any} \ h \in H \}$ is a subgroup of $H$, and, in particular, the center $Z(G)$ is contained in $H$. 
\end{remark}

\begin{proof}
Clearly, $C_{G}(H) \subset S(H, \chi)$. For an irreducible weight pair $(H, \chi)$, it follows from Schur's lemma that $\End_{G}(\pi) = \CC$, and then $S(H, \chi) = H$.
\end{proof}

Now let $G$ be a finitely generated nilpotent group.

\medskip
We will essentially use the following theorem. It allows us to replace irreducibility of representations with Schur-irreducibility, which is much easier to check. 

\medskip

\begin{theorem}\label{theorem:schur-irr}\cite[Theorem 3.14]{bel-gor} Let $H$ be a subgroup of a finitely generated nilpotent group~$G$. Let $\rho$ be an irreducible complex representation of $H$ such that the finitely induced representation $\ind_{H}^{G} (\rho)$ satisfies~$\End_{G} \big( \ind_{H}^{G} (\rho) \big) =~\CC$. Then the representation~${\ind_{H}^{G} (\rho)}$ is irreducible.
\end{theorem}

We will use this theorem for $\rho = \chi$, a one-dimensional representation of $H$. 
\begin{remark}
Under the conditions of Theorem~\ref{theorem:schur-irr}, a representation~${\ind_{H}^{G} (\rho)}$ is irreducible if and only if $S(H, \chi) = H$.
\end{remark}

\medskip


\subsection{Irreducible weight pairs with abelian subgroups} 

Let us denote by $\gamma_1(G) = [G, G]$, $\gamma_k(G) = [G, \gamma_{k-1}(G)]$.

\begin{definition}\label{def:star}
Let $H^{*}$ be the smallest subgroup of $G$ with the following properties: $H^*$ contains $H$ and if an element $g\in G$ satisfies $g^{i} \in H^*$ for some positive integer $i$, then~${g\in H^{*}}$.

A subgroup $H$ is called {\it isolated} if $H = H^*$.

\end{definition} 

For elements $g, h \in G$ we will denote a conjugated element $h g h^{-1}$ by $g^{h}$.

\begin{proposition} \label{prop:isol}
Let $G$ be a finitely generated torsion-free nilpotent group, such that torsion of $(G / \gamma_k(G))$ is trivial for every $k$.
Let $H \subset G$ be an abelian subgroup of $G$. If $H$ form an irreducible weight pair $(H, \chi)$ for some character $\chi: H \rightarrow \CC^{*}$, then $H$ is isolated.
\end{proposition}

To prove the proposition, we need the following simple lemma:

\begin{lemma}\label{lemma:hom}
Let $g$, $h$ be elements of $G$ such that $k$ is the maximal number that $ [h, g] \in \gamma_k(G)$. Then there is a homomorphism $\psi: \big< g \big> \rightarrow G / \gamma_{k+1}(G)$, which maps $g^i$ to $[h, g^i]$.
\end{lemma}

\begin{proof}
Follows directly from the formula $$ [h, g^i g^j] = [h, g^i] [h, g^j]^{g^i} = [h, g^i] [h, g^j] [g^i, [h, g^j]], 
$$
since $ [g^i, [h, g^j ]]$ is in $\gamma_{k+1}(G)$.
\end{proof}

Now we can prove the Proposition~\ref{prop:isol}.

\begin{proof}
Assume the opposite, then there exists an element $g \in H^{*} \setminus H$. That is $g^n \in H$ for some integer $n$. Since $C_{G}(H) = H$, there exists an element $h \in H$ such that $g$ does not commute with $h$. Since $G$ is nilpotent, there exists $k$ such that $[h,g] \in \gamma_k(G)$, and $[h,g] \notin \gamma_{k+1}(G)$. Note that $k \geq 1$. By Lemma~\ref{lemma:hom} there is a homomorphism $\psi: \big< g \big> \rightarrow \gamma_{k}(G) / \gamma_{k+1}(G)$, but since the torsion of $\gamma_{k}(G) / \gamma_{k+1}(G)$ is trivial, $\psi$ is injective. Since $g^n \in H$, it follows that $[h, g^n] = \psi(g^n) = 1$. But $g^n \ne 1$, because $g \ne 1$, and $G$ is torsion-free. It contradicts injectivity of $\psi$. 

\end{proof}

\subsection{Isomorphic finitely induced representations.}

We have the following criterion of isomorphism of irreducible monomial representations.

\begin{definition}\label{z-equiv}
We say that pairs $(H_1, \chi_1)$ and $(H_2, \chi_2)$ are {\it equivalent} if~${\ind_{H_1}^G(\chi_1)\simeq \ind_{H_2}^G(\chi_2)}$. We denote it as follows: $(H_1, \chi_1) \sim (H_2, \chi_2)$.
\end{definition}

\begin{proposition}\label{prop:twosubgroups}\cite[Proposition 4.10]{bel-gor}
Let~${(H_1,\chi_1)}$ and~${(H_2,\chi_2)}$ be two irreducible weight pairs. Then they are equivalent if and only if there exists an element $g \in G$ such that~${(H_2^*)^g=H_1^*}$ and ${\chi_1|_{H_2^g\cap H_1} = \chi_2^g|_{H_2^g\cap H_1}}$.

\end{proposition} 

Let $Y$ be the set of irreducible weight pairs. Let $\Sigma$ be the set of all subgroups $H \subset G$ for which there exists a character $\chi : H \rightarrow \CC^*$ such that a pair $(H, \chi)$ belongs to $Y$. One has a natural surjective map $Y \rightarrow \Sigma$.

\medskip

Let us denote by $\sim_{f}$ the equivalence~\ref{z-equiv}, restricted on fibers of the map $Y \rightarrow \Sigma$. We denote the quotient by this equivalence by $Z_{H} = Y_{H} / \sim_{f}$, where $Y_{H}$ is a fiber of $Y$ over a subgroup $H \in \Sigma$. We denote by $Z$ the bundle over $\Sigma$ with fibres $Z_H$ over a subgroup $H \in \Sigma$.

Let us denote by $W$ the set of equivalence classes of irreducible weight pairs such that $(H_1, \chi_1)$ and $(H_2, \chi_2)$ belong to the same class if the pairs are equivalent, and $H_1^* = H_2^*$. Let us denote this equivalence by $\sim_{*}$. This equivalence relation on $Y$ naturally descends to the set of subgroups $\Sigma$. Subgroups $H_1$, $H_2 \in \Sigma$ belong to the same equivalence class if $H_1^* = H_2^*$. Let us denote the quotient og $\Sigma$ by this equivalence by $\Theta$. There is a natural surjective map $W \rightarrow \Theta$.




\medskip

\begin{corollary}\label{spaces}
We have the following commutative diagram:

\[\begin{tikzcd}
 Y \arrow{r} \arrow[swap]{d} & Z \arrow{r}{\varphi} \arrow[swap]{d} & W  \arrow{d} \\
 \Sigma \arrow{r}{\sim} & \Sigma \arrow{r} & \Theta 
\end{tikzcd} 
\]

\end{corollary}


\begin{definition}
A weight pair $(H', \chi')$ {\it extends} a weight pair $(H, \chi)$ if $H$ is a subgroup of $H'$ and $\chi' \vert_{H} = \chi$.
\end{definition}

\begin{proposition}\label{lem:isopairs}
Let $G$ be a finitely generated nilpotent group. Then for every representative of equivalent weight pairs $(H, \chi) \in W$ the set of equivalence classes of weight pairs $\varphi^{-1}\big( (H, \chi) \big)$ in $Z$ is finite.

\end{proposition}

\begin{proof}
 
Let $(H, \chi)$ be an element of $W$. Let us prove that there are finitely many equivalence classes of weight pairs $(H', \chi') \in Z$ such that $(H', \chi') \sim (H, \chi)$ and $H'^* = H^*$.

Since $G$ is finitely generated nilpotent group, $G$ is Noetherian. Then every subgroup $H \subset G$ is finitely generated. Let $H = \big< h_1, \dots, h_k \big>$. Since $H'^* = H^*$, there exist certain $n_i \in \NN$, $m_i \in \ZZ$ such that $H' = \big< h_1^{\frac{n_1}{m_1}}, \dots, h_k^{\frac{n_k}{m_k}} \big>$. For any element $h \in H$ there are finitely many $g \in G$ such that $g^r = h$ for some integer $r$, since for any subgroup $H$ its index in $H^*$ is finite. Thus, for $h_i$, $1 \leq i \leq k$, there are finitely many integer valued tuples $(m_1^j, \dots, m_k^j)$ such that $h_i^{\frac{1}{m_1^j}} \in G$ for all indices $i, j$. Let us fix a tuple $(m_1^j, \dots, m_k^j)$ from this finite set. Let us assume that there exist infinitely many natural valued tuples $(n_1^j, \dots, n_k^j)$ such that a subgroup $H' = \big< h_1^{\frac{n_1}{m_1}}, \dots, h_k^{\frac{n_k}{m_k}} \big>$ belongs to $\Sigma$, and there exists a character $\chi'$ of $H'$ that $(H', \chi') \in Z$ and $(H', \chi') \sim (H, \chi)$. 

We call two tuples $(n_1^i, \dots, n_k^i)$ and $(n_1^j, \dots, n_k^j)$ comparable if $n_l^i \leq n_l^j$ for all indices $1 \leq l \leq k$, or $n_l^i \geq n_l^j$ for all indices $1 \leq l \leq k$. We claim that in an infinite set of natural valued tuples there always exist two comparable tuples. Indeed, for any given tuple $(n^1_1, \dots, n^1_k)$ let us find a tuple $(n^2_1, \dots, n^2_k)$ which is incomparable with $(n^1_1, \dots, n^1_k)$. It means that there exists at least one (but not all) index $1 \leq l \leq k$ such that $n^2_l < n^1_l$ and other indices $1 \leq r \leq k$ that $n^2_r \geq n^1_r$. Let us construct the next tuple which is incomparable with the first one and the second one. Then there exists an index $1 \leq s \leq k$ which correspond to the value of the third tuple which is strictly smaller than corresponding value of the second tuple. This index $s$ either coincides with $l$ or does not. The set of different natural numbers $n_1, \dots, n_k$ such that $n_1 \leq n^1_1, \dots, n_k \leq n^1_k$ is finite. If $l \ne s$ then there exists an index with strictly smaller value than of the first tuple, and we choose it from the finite set of values. If $l = s$ and the corresponding value of the third tuple $(n^3_1, \dots, n^3_k)$ coincide in $l = s$ with the value of the second tuple, that is  $n^1_l < n^3_l = n_2^l$, then we proceed by induction on $k$. Since for this case the value in $l = s$ is fixed, we are now dealing with tuples of $(k-1)$ size. It is easy to check that for $k=2$ the claim is true : there is no infinite set of incomparable natural valued tuples of the form $(n_1^j, n_2^j)$, hence the base of induction is valid. 

Thus, in our infinite set of tuples $(n_1^j, \dots, n_k^j)$ which correspond to subgroups $ \big< h_1^{\frac{n^j_1}{m_1}}, \dots, h_k^{\frac{n^j_k}{m_k}} \big>$, there exist two comparable tuples $(n_1^j, \dots, n_k^j)$ and $(n_1^i, \dots, n_k^i)$. Let us denote them by $ H_i = \big< h_1^{\frac{n^i_1}{m_1}}, \dots, h_k^{\frac{n^i_k}{m_k}} \big>$ and $ H_j = \big< h_1^{\frac{n^j_1}{m_1}}, \dots, h_k^{\frac{n^j_k}{m_k}} \big>$. Then there exist characters $\chi_i$ of $H_i$ and $\chi_j$ of $H_j$ such that weight pairs $(H_i, \chi_i)$ and $(H_j, \chi_j)$ are in $Z$, and $(H_i, \chi_i) \sim (H_j, \chi_j)$. Without loss of generality let $(n_1^i, \dots, n_k^i) < (n_1^j, \dots, n_k^j)$, then the weight pair $(H_j, \chi_j)$ extends the weight pair $(H_i, \chi_i)$. It contradicts the fact that both of them are irreducible weight pairs. Then the set of tuples is finite, and correspondingly, the set of weight pairs $(H', \chi') \in Z$ such that $(H', \chi') \sim (H, \chi)$ and $H'^* = H^*$ is also finite.

\end{proof}

For a given weight pair $(H, \chi) \in Y$, if the quotient $G / N_{G}(H^{*})$ is non-trivial, then it is infinite. For any element $g \in G \setminus N_{G}(H^{*})$
the weight pair $(H^g, \chi^g)$ is irreducible and equivalent to the pair $(H, \chi)$. Since conjugation on irreducible weight pairs by elements of $G$ commutes with the mapping to a $\sim_{*}$-equivalent weight pair, we can take the consequent quotients of $Y$ by these equivalences. We denote by $X$ the quotient of $Y$ by this equivalences. 

The latter equivalence relation naturally descends to the set of subgroups~$\Sigma$. Subgroups $H_1$, $H_2 \in \Sigma$ belong to the same equivalence class if there exists an element $g \in G$ such that $(H_1^*)^g = H_2^*$. We denote by $\Xi$ the quotient by this equivalence of $\Theta$. There is a natural surjective map $X \rightarrow \Xi$.



\medskip

Thus, we have the following commutative diagram:

\[\begin{tikzcd}
 Y \arrow{r} \arrow[swap]{d} & Z \arrow{r}{\varphi} \arrow[swap]{d} & W \arrow{r} \arrow[swap]{d} & X \arrow{d} \\
 \Sigma \arrow{r}{\sim} & \Sigma \arrow{r} & \Theta \arrow{r} & \Xi
\end{tikzcd} 
\]

\subsection{Ranks of finitely generated nilpotent groups}

\begin{definition}
Let $G$ be a finitely generated nilpotent group.
For a given subgroup $H \subset G$ we inductively define $\mathrm{rk}_{1}(G):= \mathrm{rk}(H / H \cap [G, G])$, $\mathrm{rk}_{i+1} (H):= \mathrm{rk}((H \cap \gamma_{i}(G) ) / ( H \cap \gamma_{i+1}(G)))$.

\end{definition}

\begin{remark}

If $(H_1, \chi_1)$ and $(H_2, \chi_2)$ are equivalent irreducible weight pairs, then $\mathrm{rk}_{i}(H_1) = \mathrm{rk}_{i}(H_2) $ for all $i$.

\end{remark}

\begin{proof}

If $H_1^{*} = H_2^{*}$ then the ranks are clearly equal since ranks do not change if one restricts to a subgroup of a finite index.

Let $(H_{1}^{*})^{g} = H_2^{*}$ for some non-trivial $g \in G$. Let $h_1, \dots, h_n$ generate the quotient $(H_1 \cap \gamma_i(G)) / (H_1 \cap \gamma_{i+1}(G))$. Since $ (\big < h_1^g, \dots, h_n^g \big > \cap \gamma_i(G)) / (H_1 \cap \gamma_{i+1}(G))$ coincides with the quotient $(\big < h_1, \dots, h_n \big > \cap \gamma_i(G)) / (H_1 \cap \gamma_{i+1}(G))$, we have $\mathrm{rk}_{i}(H_1) = \mathrm{rk}_{i}(H_2) $ for all $i$.

\end{proof}

\section{Classification of irreducible weight pairs} \label{sect:classification}

Let $G$ be the group of upper triangular matrices of the fourth order with integer entries. We will classify all irreducible weight pairs $(H, \chi)$ such that $H \subset G$. 

\begin{proposition} 

[(i)] The ranks of subgroups of the group $G$ can be the following: $(0,0)$, $(0,1)$, $(0,2)$, $(1,0)$, $(1,1)$, $(1,2)$, $(2,0)$, $(2,1)$, $(2,2)$ and $(3,2)$. 

[(ii)] If $H \in \Sigma$, then ranks of $H$ can be the following: $(1,1)$, $(2,0)$, $(2,1)$, $(1,2)$, $(2,2)$ and $(3,2)$. 

\end{proposition}

\begin{proof}

[(i)] If $\mathrm{rk}_1(H) = 3$, then there are two generators of $H$ in $H \cap [G, G] / Z(G)$, so $\mathrm{rk}_2(H) = 2$. Thus, cases $(3,0)$ and $(3,1)$ are not possible.

[(ii)] It is easy to check that for any $g \in G$ the centralizer $C_{G}(g)~\supsetneq~\big<g, Z(G) \big>$. For an irreducible weight pair $(H, \chi)$ we have $S(H, \chi)~=~H$, and by Remark~\ref{rem-cent} for $g \in H$ a subgroup $\big<g, Z(G) \big> \subsetneq H$. Then there are always more than two generators in $H$. Thus, the cases $(0,0)$, $(0,1)$, and $(1,0)$ are not possible. The case of $(0,2)$ ranks is also not possible since the centralizer $C_{G}(H)$ coincides with the following subgroup of~$G$:
$$
\begin{pmatrix}
     1 & 0 &  \ZZ  &  \ZZ  \\
     0 & 1 &  \ZZ  &  \ZZ  \\
     0 & 0 & 1 & 0 \\
     0 & 0 & 0 & 1
\end{pmatrix},
$$
which is bigger than $H$. Then there exists a weight pair $(H', \chi')$ which extends the pair $(H, \chi)$. It contradicts the condition $S(H, \chi)~=~H$.
\end{proof}

Let us denote by $\Sigma_{r_1, r_2} $ the restriction of $\Sigma$ to the set of subgroups $H \subset G$ such that $\rk_1(H) =~r_1$, $\rk_2(H) = r_2$. The set $\Sigma$ is a disjoint union of sets $\Sigma_{r_1, r_2}$. Let us denote corresponding bundles $Y$, $Z$ restricted to $\Sigma_{r_1, r_2} $ by $Y_{r_1, r_2}$, $Z_{r_1, r_2}$, and similarly the bundle $X_{r_1, r_2} ~\rightarrow~\Xi_{r_1, r_2}$.

\medskip 
In this subsection we describe consecutively the set $\Sigma_{r_1, r_2} $, bundles $Y_{r_1, r_2}$, $Z_{r_1, r_2}$, and $X_{r_1, r_2} ~\rightarrow~\Xi_{r_1, r_2}$ for all possible ranks of subgroups of the set $\Sigma$ for the group $G$. 

\medskip

We denote by $\mu_{\infty}$ the union of groups of roots from unity.

\section{ \texorpdfstring{$\text{The case of }\mathrm{rk}_{1}(H) = 1, \ \mathrm{rk}_{2}(H) = 1.$}{Case 1}}
\bigskip


Let us define sets $S, S_1, S_2, S_3, S_4, N_1$, and $N_2$.

$$S = \{ (a, d, f, b, e) \in \ZZ^{5} \ \vert \ \mathrm{GCD} \,(\frac{f b - a e}{n}, a, d, f) = 1, \ \ n = \mathrm{GCD}(a, f) \},$$

$$S_1 = \{ (a, d, f, b, e) \in \ZZ^{5} \ \vert \ a \neq 0, \ d \neq 0, \ f \neq 0  \} \ \cap \ S,$$

$$S_2 = \{ (a, d, f, b, e) \in \ZZ^{5} \ \vert   \ a = 0, \ d \neq 0, \ f \neq 0 \}  \ \cap \ S,$$

$$S_3 = \{ (a, d, f, b, e) \in \ZZ^{5} \ \vert \ a \neq 0, \ d \neq 0, \ f = 0  \}  \ \cap \ S,$$

$$S_4 = \{ (a, d, f, b, e) \in \ZZ^{5} \ \vert \ a \neq 0, \ d = 0, \ f \neq 0 \}  \ \cap \ S,$$

$$N_1 = \{ (a, d, f, b, e) \in \ZZ^{5} \ \vert  \ a \neq 0, \ d = 0, \ f = 0, \ b = 1 \}  \ \cap \ S,$$ 

$$N_2 = \{ (a, d, f, b, e) \in \ZZ^{5} \ \vert \ a = 0, \ d = 0, \ f \neq 0, \ e = 1  \}  \ \cap \ S.$$

\begin{proposition}\label{prop:1}
There is a canonical bijection $\phi$ from $S_1 \ \cup \ S_2 \ \cup \ S_3 \ \cup \ S_4 \ \cup \ N_1  \ \cup \ N_2$ to $\Sigma_{1, 1}$. It maps a tuple $(a, d, f, b, e)$ to a subgroup $H$, generated by the following matrices $$
h_1 =
\begin{pmatrix}
     1 & a & b & 0 \\
     0 & 1 & d & e \\
     0 & 0 & 1 & f \\
     0 & 0 & 0 & 1
\end{pmatrix},
\quad
h_2 =
\begin{pmatrix}
     1 & 0 & \frac{a}{n} & 0 \\
     0 & 1 & 0 & \frac{f}{n} \\
     0 & 0 & 1 & 0 \\
     0 & 0 & 0 & 1
\end{pmatrix},
\quad
C =
\begin{pmatrix}
     1 & 0 & 0 & 1 \\
     0 & 1 & 0 & 0 \\
     0 & 0 & 1 & 0 \\
     0 & 0 & 0 & 1
\end{pmatrix},
$$
where $n = \mathrm{GCD}(a, f)$. Moreover, for every $H \in \Sigma_{1, 1}$, a subgroup $H$ is abelian.

We can extend the bijection $\phi$ to the map from $$ (\CC^{*})^2 \times \{ \CC^{*} \setminus \mu_{\infty} \} = \{ t, z \in (\CC^{*})^2, \ \lambda \notin \mu_{\infty} \}$$ to $Y_{1,1;\, H}$, which is 
defined as follows: $$ t = \chi(h_1) , \ \ z = \chi(h_2) , \ \  \lambda = \chi(C).$$
\end{proposition}

\begin{proof}

The proof goes as follows. First, we prove that if $H \in \Sigma_{1, 1}$, then a subgroup $H$ is abelian, and it is generated by $h_1, h_2, C$ in the form given in Proposition~\ref{prop:1}. Then we obtain conditions for generators $h_1, h_2$ that the corresponding subgroup $H$ can form an irreducible weight pair with some character $\chi$. Then we study characters $\chi_1, \chi_2: H \rightarrow \CC^{*}$ which correspond to equivalent irreducible weight pairs $(H, \chi_1)$ and $(H, \chi_2)$. Finally, we obtain conditions for a character $\chi$ that $(H, \chi)$ is an irreducible weight pair.

 Since $ \mathrm{rk}_{1}(H) = \mathrm{rk}_{2}(H) =  1$, we have one generator $h_1$ such that $ \big< h_1 \big> / \big( [G, G] \cap H\big)$ is not trivial. Let us denote by $h_2$ a generator such that $ \big< h_2 \big> / \big( [G, G] \cap H \big)$ is trivial and $ \big( \big< h_2 \big> \cap [G, G] \big) / Z(G)$ is not trivial. 
By Remark~\ref{rem-cent} the center $Z(G)$ is contained in $H$, hence we can choose generators $h_1$ and $h_2$ such that $ \big< h_1 \big> \cap Z(G)$ and $ \big< h_2 \big> \cap Z(G)$ are trivial.

If the commutator $[h_1, h_2]$, which is contained in the center of $G$, is not trivial, then $\chi( [h_1, h_2] ) = \chi(C)^N = 1$ for some integer $N$. But if $\chi(C)$ is a root of unity, then $\rk_2(C_{G}(H)) = 2$, since in this case both elements
$$
\begin{pmatrix}
     1 & 0 & k_1 & 0 \\
     0 & 1 & 0 & 0 \\
     0 & 0 & 1 & 0 \\
     0 & 0 & 0 & 1
\end{pmatrix},
\quad 
\begin{pmatrix}
     1 & 0 & 0 & 0 \\
     0 & 1 & 0 & k_2 \\
     0 & 0 & 1 & 0 \\
     0 & 0 & 0 & 1
\end{pmatrix}
$$
are contained in $H$ for some integers $k_1$ and $k_2$. By Remark~\ref{rem-cent} the centralizer $C_{G}(H) \subset H$, thus $\chi(C)$ is not a root of unity.

Hence, generators $h_1, h_2$ commute, and $H$ is abelian. Then it follows from Lemma~\ref{prop:isol} that $H = H^{*}$. 

Put
$$
h_1 =
\begin{pmatrix}
     1 & a & b & 0 \\
     0 & 1 & d & e \\
     0 & 0 & 1 & f \\
     0 & 0 & 0 & 1
\end{pmatrix}.
$$

If $a = f = 0$, then $\rk_2(C_{G}(H)) = 2$, but the centralizer $C_{G}(H) \subset H$ by Remark~\ref{rem-cent}. Then either $a$ or $f$ are non-zero.

Since generators $h_1, h_2$ commute, we obtain that $h_2^{k}$ is equal to
$$
\begin{pmatrix}
     1 & 0 & a & 0 \\
     0 & 1 & 0 & f \\
     0 & 0 & 1 & 0 \\
     0 & 0 & 0 & 1
\end{pmatrix}
$$
for some integer $k$.
If $a$ and $f$ are coprime then $h_2$ coincides with this element. If not, then since a subgroup $H$ is isolated, it has to contain all its roots. Hence
$$
h_2 =
\begin{pmatrix}
     1 & 0 & \frac{a}{n} & 0 \\
     0 & 1 & 0 & \frac{f}{n} \\
     0 & 0 & 1 & 0 \\
     0 & 0 & 0 & 1
\end{pmatrix},
$$
where $n = \mathrm{GCD}(a, f)$. Let us denote by $(a', f')$ the proportional coprime tuple~$( \frac{a}{n},  \frac{f}{n})$. 

Now we will obtain conditions for a subgroup $H$ to be an isolated subgroup. In order to do that we will study tuples $(a, d, f, b, e)$ which correspond to generators of subgroups $H$ such that $H \subset H^{*}$.

Let $g \in H^{*} \setminus H$. That means that there are such natural $r_1 > 1$ and integer $r_2$, $r_3$ that $g^{r_1} = h_{1}^{r_2} h_{2}^{r_3} $. Let 
$$
g =
\begin{pmatrix}
     1 & \tilde{a} & \tilde{b} & 0 \\
     0 & 1 & \tilde{d} & \tilde{e} \\
     0 & 0 & 1 & \tilde{f} \\
     0 & 0 & 0 & 1
\end{pmatrix}.
$$

That means that 

$$
\begin{pmatrix}
     1 & r_1 \tilde{a} & r_1\tilde{b} + \frac{r_1 (r_1 - 1)}{2} \tilde{a} \tilde{d}  & 0 \\
     0 & 1 & r_1 \tilde{d} & r_1\tilde{e} + \frac{r_1 (r_1 - 1)}{2} \tilde{f} \tilde{d} \\
     0 & 0 & 1 &r_1 \tilde{f} \\
     0 & 0 & 0 & 1
\end{pmatrix}
=
$$

$$
\begin{pmatrix}
     1 & r_2 a & r_2 b + \frac{r_2 (r_2 - 1)}{2} a d  & 0 \\
     0 & 1 & r_2 d & r_2 e + \frac{r_2 (r_2 - 1)}{2} f d \\
     0 & 0 & 1 &r_2 f \\
     0 & 0 & 0 & 1
 \end{pmatrix}    
 \times
\begin{pmatrix}
     1 & 0 & r_3 a' & 0 \\
     0 & 1 & 0 & r_3 f' \\
     0 & 0 & 1 & 0 \\
     0 & 0 & 0 & 1
\end{pmatrix}
$$

It implies that 

$$
\begin{pmatrix}
     r_1 \tilde{a}  \\
     r_1 \tilde{d}  \\
     r_1 \tilde{f} 
\end{pmatrix}
 =
 \begin{pmatrix}
     r_2 a  \\
     r_2 d  \\
     r_2 f
\end{pmatrix}
\text{  \ and \ }
\begin{pmatrix}
     r_1 \tilde{a}  \\
     r_1 \tilde{d}  \\
     r_1 \tilde{f} 
\end{pmatrix}
= l r_1 \begin{pmatrix}
    a''  \\
    d''  \\
      f'' 
\end{pmatrix},
$$
where $(a'', d'', f'')$ is a coprime tuple which is proportional to $(a, d, f)$. 

Then $r_1 l = r_2 n$ for some integer $l$. Let us notice that if $n = 1$ then it is easy to observe that $g$ is an element of $H$.

If $\mathrm{GCD} (l, n) $ is greater than $1$, then we may divide the equality $r_1 l = r_2 n$ by $\mathrm{GCD} (l, n)$ and obtain the new equality $r_1 l' = r_2 n'$ with 
$\mathrm{GCD} (l', n') = 1$. After that, we can replace a tuple $(l, n)$ with a proportional coprime tuple $(l', n')$. Since modulo this replacement it does not change the proof, we proceed with the case when $\mathrm{GCD} (l, n) = 1$.

We have $r_1 = n r$ and $r_2 = l r$ for some integer $r$. 

From $g^{r_1} = h_{1}^{r_2} h_{2}^{r_3} $ we obtain $$
r_1\tilde{b} + \frac{r_1 (r_1 - 1)}{2} l^2 a'' d'' = r_2 b + \frac{r_2 (r_2 - 1)}{2} n^2 a'' d'' + r_3 a', 
$$
$$
r_1\tilde{e} + \frac{r_1 (r_1 - 1)}{2} l^2 f'' d'' = r_2 e + \frac{r_2 (r_2 - 1)}{2} n^2 f'' d'' + r_3 f'.
$$

Then $$
r_1\tilde{b} -  r_2 b =  (\frac{r_2 (r_2 - 1)}{2} n^2 k d'' -  \frac{r_1 (r_1 - 1)}{2} l^2 k d''  + r_3) a', 
$$
$$
r_1\tilde{e} -  r_2 e =  (\frac{r_2 (r_2 - 1)}{2} n^2 k d'' -  \frac{r_1 (r_1 - 1)}{2} l^2 k d''  + r_3) f'.
$$ 

Let us denote the expression $(\frac{r_2 (r_2 - 1)}{2} n^2 k d'' -  \frac{r_1 (r_1 - 1)}{2} l^2 k d''  + r_3)$ by $M$. Then $M$ can be an arbitrary integer since $r_3$ can be an arbitrary integer.
Now we need to obtain the condition for the system of equations 

$$
r(n \tilde{b} - l b) = M a',
$$
$$
r(n \tilde{e} - l e) = M f'.
$$
to be solvable in $\tilde{b}, \tilde{e} \in \ZZ^2$, where integers $M, r$, and $l$ are coprime with $n$, and integers $b, e, n, a', f'$ are fixed.

The system is solvable if and only if $$
l (b f' - e a') = n (\tilde{b} f' - \tilde{e} a').
$$

Since $n$ and $l$ are coprime, the condition above is equivalent to the following one: $(b f' - e a')$ is divisible by $n$. Then $(\tilde{b} f' - \tilde{e} a') = l m $ for some integer $m$ and since $a'$ and $f'$ are coprime, there always exist integers $\tilde{b}$ and $\tilde{e}$ which satisfy the equation. 

Finally, we can see that a subgroup $H$ is isolated if and only if $H \in \phi(S)$.
 
\bigskip

Now for every $H \in \Sigma_{1, 1}$, we will describe the fiber $Y_{1, 1; \, H}$. We need to check that $Y_{1, 1; \, H}$ parametrizes all characters $\chi$ of $H$ which correspond to irreducible weight pairs $(H, \chi)$.

\medskip

{\bf (i)}. If $H \in \Sigma_{1, 1} \setminus ( \phi(N_1)  \ \cup \ \phi(N_2) \ \cup \ \phi(S_2)  \ \cup \ \phi(S_3))$ then the quotient $N_{G}(H) / H$ is generated by

$$
g_1 =
\begin{pmatrix}
     1 & a' & 0 & 0 \\
     0 & 1 & 0 & 0 \\
     0 & 0 & 1 & -f' \\
     0 & 0 & 0 & 1
\end{pmatrix}, 
\quad
g_2 =
\begin{pmatrix}
     1 & 0 & 0 & 0 \\
     0 & 1 & 0 & 1 \\
     0 & 0 & 1 & 0 \\
     0 & 0 & 0 & 1
\end{pmatrix}.
$$

The action of $g_1$ and $g_2$ on a character $\chi$ is as follows: 

\begin{equation} \label{eq:1notnormal}
\chi^{g_1}(h_1) = \chi(h_1) \chi(h_2)^{d} \lambda^{a'e + f'b+ a'f'd}, \ \ \  \chi^{g_1}(h_2) = \chi(h_2) \lambda^{2a'f'}, \\
\end{equation}
\begin{center}
and
\end{center}
$$\chi^{g_2}(h_1) = \chi(h_1) \lambda^{-a},  \ \ \  \   \chi^{g_2}(h_2) = \chi(h_2).$$

We can see that if $\chi(C) = \lambda$ is not a root of unity, then the action above is free, which means that $S(H, \chi) = H$ for a corresponding weight pair $(H, \chi)$.

If $H \in \phi(S_2) $, then the generator $g_2$ is replaced by

$$
g_2 =
\begin{pmatrix}
     1 & 0 & 1 & 0 \\
     0 & 1 & 0 & 0 \\
     0 & 0 & 1 & 0 \\
     0 & 0 & 0 & 1
\end{pmatrix}.
$$
 The action is as follows:
 
\begin{equation} \label{eq:1notnormal}
\chi^{g_1}(h_1) = \chi(h_1) \chi(h_2)^{d} \lambda^{f'b}, \ \ \  \chi^{g_1}(h_2) = \chi(h_2), \\
\end{equation}
\begin{center}
and
\end{center}
$$\chi^{g_2}(h_1) = \chi(h_1) \lambda^{f},  \ \ \  \   \chi^{g_2}(h_2) = \chi(h_2).$$

Again, we can see that if $\chi(C) = \lambda$ is not a root of unity, then the action above is free, which means that $S(H, \chi) = H$.

{\bf (ii)}. If $H \in (\phi(N_1)  \ \cup \  \phi(N_2))$, then it is easy to check that $N_{G}(H) = G$. Without loss of generality let us consider the case of a subgroup $H$ with $f = d = 0$. The other one with $a = d = 0$ is treated similarly. Since $H$ is isolated, its generators $h_1$ and $h_2$ may be chosen as follows: 
$$
h_1 =
\begin{pmatrix}
     1 & a & 0 & 0 \\
     0 & 1 & 0 & e \\
     0 & 0 & 1 & 0 \\
     0 & 0 & 0 & 1
\end{pmatrix}, 
\quad
h_2 =
\begin{pmatrix}
     1 & 0 & 1 & 0 \\
     0 & 1 & 0 & 0 \\
     0 & 0 & 1 & 0 \\
     0 & 0 & 0 & 1
\end{pmatrix}.
$$

Then the quotient $N_{G}(H) / H$ is generated by 
$$
g_1 =
\begin{pmatrix}
     1 & 0 & 0 & 0 \\
     0 & 1 & 1 & 0 \\
     0 & 0 & 1 & 0 \\
     0 & 0 & 0 & 1
\end{pmatrix}, 
\quad
g_2 =
\begin{pmatrix}
     1 & 0 & 0 & 0 \\
     0 & 1 & 0 & 1 \\
     0 & 0 & 1 & 0 \\
     0 & 0 & 0 & 1
\end{pmatrix}
$$
\begin{center}
and 
\end{center}
$$
g_3 =
\begin{pmatrix}
     1 & 0 & 0 & 0 \\
     0 & 1 & 0 & 0 \\
     0 & 0 & 1 & 1 \\
     0 & 0 & 0 & 1
\end{pmatrix}.
$$

The action of $g_1, g_2$ and $g_3$ on characters is as follows:

\begin{equation} \label{eq:1normal}
\begin{split}
\chi^{g_1}(h_1) = \chi(h_1) \chi(h_2)^{-a}, \ \ \  \chi^{g_1}(h_2) = \chi(h_2),  \\
\chi^{g_2}(h_1) = \chi(h_1) \chi(C)^{-a},  \ \ \  \   \chi^{g_2}(h_2) = \chi(h_2), \\
\chi^{g_3}(h_1) = \chi(h_1),  \ \ \  \   \chi^{g_3}(h_2) = \chi(h_2) \chi(C)^{-1}.
\end{split}
\end{equation}

Again, we can see that if $\chi(C) = \lambda$ is not a root of unity,  then the action above is free, which means that $S(H, \chi) = H$ and corresponding representations are irreducible.

\end{proof}

Let us consider the following action: $z \rightarrow z \lambda^n, \ n \in \ZZ$.

If $\lambda$ is a root of unity, then the quotient of $z \in \CC^{*}$ by this action is conformally equivalent to $\CC^{*}$.

In a case that $\lambda$ is not in $S^1$, then the quotient by this action is an elliptic curve, which we denote by $E_{\lambda} = \big< \CC^{*} / \lambda^n, \ n \in \ZZ \big>$. 

In a case that $\lambda \in S^1 \setminus \mu_{\infty}$, we denote the corresponding quotient by $P_{\lambda}$ (not separable).

Let both $\lambda$ and $z$ be not roots of unity. We denote by $T_{z, \lambda} = \big< \CC^{*} / \lambda^{n_1} z^{n_2}, \ n_1, n_2 \in \ZZ^2 \big>$ (not separable).

\medskip

Let us denote by $Z_{1, 1;\, H}^{(3)} = \{ \{ t, z \in (\CC^{*})^2, \ \lambda \notin \mu_{\infty} \} / \sim \}$, by $ Z_{1, 1;\, H}^{(2)} =  \{ \{ z \in \CC^{*}, \ \lambda \notin \mu_{\infty} \} / \sim \}$ and by $Z_{1, 1;\, H}^{(1)} = \{  \lambda \notin \mu_{\infty} \} $, where $\sim$ is defined in~\ref{z-equiv}.

\medskip

\begin{corollary}\label{corollary:fib1}

If $H \in \phi(S_1)$, then the fiber $Z_{1, 1;\, H}$ of  $Z_{1, 1}$ over a subgroup $H$ has iterated structure of a bundle, namely:
$$Z_{1, 1;\, H}^{(3)} \rightarrow Z_{1, 1;\, H}^{(2)} \rightarrow Z_{1, 1;\, H}^{(1)}.$$ We describe fibers of this bundle consecutively in coordinates $(t, z, \lambda)$. 

$$\{ T_{ {z^d \lambda^{a'e + f'b + a'f'd}, \ a} }\} \ , \ \{E_{\lambda^{2 a' f' }} \setminus \mu_{\infty}\} \ , \ \{\CC^* \setminus S^1\} \ \cup $$
$$ \{ T_{z^d \lambda^{a'e + f'b + a'f'd}, \ a}  \} \ , \ \{P_{\lambda^{2 a' f' }} \setminus \mu_{\infty}\} \ , \ \{S^1 \setminus \mu_{\infty}\} 
 \ \cup $$
 $$ \{ E_{\lambda^{\mathrm{GCD}(a'e + f'b + a'f'd, a)}} \} \ , \ \{ \mu_{\infty} \} \ , \ \{\CC^* \setminus S^1\} \ \ \cup \ \  \{ P_{\lambda^{\mathrm{GCD}(a'e + f'b + a'f'd, a)}} \} \ , \ \{ \mu_{\infty} \} \ , \ \{S^1 \setminus \mu_{\infty}\};$$
 
\medskip

If $H \in \phi(S_2)$, then the fibers $Z_{1, 1;\, H}^{(3)} \rightarrow Z_{1, 1;\, H}^{(2)} \rightarrow Z_{1, 1;\, H}^{(1)}$ over $H$ are canonically bijective to: 
$$ \{ T_{z^d \lambda^{f'b}, \lambda^{f}} \} \ , \ \{\CC^* \setminus S^1\} \ , \ \{\CC^* \setminus S^1\}  \ \cup \  \{ T_{z^d \lambda^{f'b}, \lambda^{f}} \} \ , \ \{S^1 \setminus \mu_{\infty}\} \ , \ \{\CC^* \setminus S^1\}  \ \cup $$
$$\{ E_{\lambda^{\mathrm{GCD}(f'b, f)}} \} \ , \ \{ \mu_{\infty} \} \ , \ \{\CC^* \setminus S^1\}  \ \cup \  \{ T_{z^d \lambda^{f'b}, \lambda^{f}} \} \ , \ \{\CC^* \setminus \mu_{\infty}\} \ , \ \{S^1 \setminus \mu_{\infty}\}
\ \cup $$
$$ \{ P_{\lambda^{\mathrm{GCD}(f'b, f)}}  \} \ , \ \{ \mu_{\infty} \} \ , \ \{S^1 \setminus \mu_{\infty}\};$$

 \medskip
 
If $H \in  \phi(S_3)$, then the fibers $Z_{1, 1;\, H}^{(3)} \rightarrow Z_{1, 1;\, H}^{(2)} \rightarrow Z_{1, 1;\, H}^{(1)}$ over $H$ are canonically bijective to: 
$$\{ T_{z^d \lambda^{a'e}, \lambda^{a}} \} \ , \ \{\CC^* \setminus S^1\} \ , \ \{\CC^* \setminus S^1\}  \ \cup \  \{ T_{z^d \lambda^{a'e}, \lambda^{a}} \} \ , \ \{S^1 \setminus \mu_{\infty}\} \ , \ \{\CC^* \setminus S^1\}  \ \cup $$
$$\{ E_{\lambda^{\mathrm{GCD}(a'e, a)}} \} \ , \ \{ \mu_{\infty} \} \ , \ \{\CC^* \setminus S^1\}  \ \cup \  \{ T_{z^d \lambda^{a'e}, \lambda^{a}} \} \ , \ \{\CC^* \setminus \mu_{\infty}\} \ , \ \{S^1 \setminus \mu_{\infty}\}
\ \cup $$
$$ \{ P_{\lambda^{\mathrm{GCD}(a'e, a)}} \} \ , \ \{ \mu_{\infty} \} \ , \ \{S^1 \setminus \mu_{\infty}\};$$

 \medskip
 
If $H \in \phi(S_4)$, then the fibers $Z_{1, 1;\, H}^{(3)} \rightarrow Z_{1, 1;\, H}^{(2)} \rightarrow Z_{1, 1;\, H}^{(1)}$ over $H$ are canonically bijective to:
$$\{E_{\lambda^{\mathrm{GCD}(a'e + f'b, a)}} \} \ , \ \{E_{\lambda^{2 a' f'}} \} \ , \ \{\CC^* \setminus S^1\}  \ \cup \  \{ P_{\lambda^{\mathrm{GCD}(a'e + f'b, a)}} \} \ , \ \{ P_{\lambda^{2 a' f'}} \} \ , \ \{S^1 \setminus \mu_{\infty}\};$$

 \medskip
 
If $H \in \phi(N_1)$, then the fibers $Z_{1, 1;\, H}^{(3)} \rightarrow Z_{1, 1;\, H}^{(2)} \rightarrow Z_{1, 1;\, H}^{(1)}$ over $H$ are canonically bijective to:
$$\{ T_{z^{a}, \lambda^{a}} \} \ , \ \{E_{\lambda} \setminus \mu_{\infty}\} \ , \ \{\CC^* \setminus S^1\}  \ \cup \  \{ E_{\lambda} \} \ , \ \{ \mu_{\infty}\} \ , \ \{\CC^* \setminus S^1\} 
 \ \cup $$
 $$ \{ T_{z^{a}, \lambda^{a}}  \} \ , \ \{ P_{\lambda} \setminus \mu_{\infty} \} \ , \ \{S^1 \setminus \mu_{\infty}\} \ \cup \  \{ P_{\lambda} \} \ , \ \{ \mu_{\infty} \} \ , \ \{S^1 \setminus \mu_{\infty}\}.$$
 
 \medskip
  
If $H \in \phi(N_2)$,  then the fibers $Z_{1, 1;\, H}^{(3)} \rightarrow Z_{1, 1;\, H}^{(2)} \rightarrow Z_{1, 1;\, H}^{(1)}$ over $H$ are canonically bijective to:
$$\{ T_{z^{f}, \lambda^{f}} \} \ , \ \{E_{\lambda} \setminus \mu_{\infty}\} \ , \ \{\CC^* \setminus S^1\}  \ \cup \  \{ E_{\lambda} \} \ , \ \{ \mu_{\infty}\} \ , \ \{\CC^* \setminus S^1\} 
 \ \cup $$
 $$ \{ T_{z^{f}, \lambda^{f}} \} \ , \ \{ P_{\lambda} \setminus \mu_{\infty} \} \ , \ \{S^1 \setminus \mu_{\infty}\} \ \cup \  \{ P_{\lambda} \} \ , \ \{ \mu_{\infty} \} \ , \ \{S^1 \setminus \mu_{\infty}\}.$$

\end{corollary} 

\begin{proof}

Follows directly from formulas of the action of $N_{G}(H) / H$ on a character $\chi$ of $H$ above (equations ~\ref{eq:1normal}, ~\ref{eq:1notnormal}).
\end{proof}

\medskip

\begin{lemma}\label{lemma:fib1}
If $H_1 \in \Sigma_{1, 1} \setminus ( \phi(N_1)  \ \cup \ \phi(N_2) )$, then the following subgroups $H_2 \in \Sigma_{1, 1}$ are equivalent to $H_1$:

$H_2  = \phi((a, d, f,  b +\tilde{d} f, e - \tilde{d} a))$ for an arbitrary integer $\tilde{d}$.

Fibers $Z_{1, 1; \, H_1}, Z_{1, 1;\, H_2} $ over subgroups $H_1, H_2$ can be canonically identified.

\end{lemma}

\begin{proof}
Since $H_1 =  H_{1}^{*}$ for $H_1 \in \Sigma_{1, 1}$, as proved in Proposition~\ref{prop:1}, we need to consider only subgroups $H_2$ such that~${(H_2)^g=H_1}$ and ${\chi_1|_{H_2^g\cap H_1} = \chi_2^g|_{H_2^g\cap H_1}}$ for some element $g\in G$. 

If $H_1 = \phi((a, d, f, b, e)) \in \Sigma_{1, 1} \setminus ( \phi(N_1)  \ \cup \ \phi(N_2))$, then $H_1$ is not normal and the quotient $G / N_{G}(H_1)$ is generated by  
$$
g = 
\begin{pmatrix}
     1 & 0 & 0 & 0 \\
     0 & 1 & 1 & 0 \\
     0 & 0 & 1 & 0 \\
     0 & 0 & 0 & 1
\end{pmatrix}.
$$
Then it is easy to compute the parameters of conjugated subgroups: $H_2~=~\phi((a, d, f,  b +\tilde{d} f, e - \tilde{d} a))$. If we denote non-central generators of $H_2$ by $\tilde{h}_1, \tilde{h}_2$, then characters of the subgroups are related as follows: $\chi_2(\tilde{h}_1) = \chi_{1}^{g}(h_1)$, $\chi_2(\tilde{h}_2) = \chi_{1}(h_2)$ and $\chi_2(C) = \chi_1(C)$. It gives a canonical identification of fibers $Z_{1, 1; \, H_1}$, $Z_{1, 1; \, H_2} $ over subgroups $H_1$ and $H_2$.

\end{proof}


\section{ \texorpdfstring{$\text{The case of }\mathrm{rk}_{1}(H) = 2, \ \mathrm{rk}_{2}(H) = 0.$}{Case 2}}
\bigskip


Let us define a set 
$$ \hspace*{-1cm}
S =  \{ (a, b, e, f',  b', e') \in \ZZ^{6} \ \vert \ ae' + f'b = 0, \ a \neq 0, \ f \neq 0, \   |\mathrm{GCD}(a, b, e) | = 1, \  \ |\mathrm{GCD}(f', b', e') | = 1 \} .$$

\medskip

\begin{lemma}

There is a canonical bijection $\phi$ from $S$ to $\Sigma_{2, 0}$. It maps a tuple $(a, b, e, f',  b', e')$ to a subgroup $H$, generated by the following matrices $$
h_1 =
\begin{pmatrix}
     1 & a & b & 0 \\
     0 & 1 & 0 & e \\
     0 & 0 & 1 & 0 \\
     0 & 0 & 0 & 1
\end{pmatrix},
\quad
h_2 =
\begin{pmatrix}
     1 & 0 & b'& 0 \\
     0 & 1 & 0 & e' \\
     0 & 0 & 1 & f' \\
     0 & 0 & 0 & 1
\end{pmatrix},
\quad
C =
\begin{pmatrix}
     1 & 0 & 0 & 1 \\
     0 & 1 & 0 & 0 \\
     0 & 0 & 1 & 0 \\
     0 & 0 & 0 & 1
\end{pmatrix}.
$$
Moreover, for every $H \in \Sigma_{2, 0}$, a subgroup $H$ is abelian.

We can extend the bijection $\phi$ to the map from $$(\CC^{*})^2 \times \{ \CC^{*} \setminus \mu_{\infty} \} = \{ (t, s, \lambda) \ \vert \ \lambda \notin \mu_{\infty} \}$$ to $Y_{2,0; \, \phi(\alpha)}$, which is 
defined as follows: $$ t = \chi(h_1) , \ \ s = \chi(h_2) , \ \  \lambda = \chi(C).$$

\end{lemma}

\begin{proof}

The proof goes as follows. First, we prove that if $H \in \Sigma_{2, 0}$, then $H$ is generated by $h_1, h_2, C$ in the form above. After that, we prove that $H$ is abelian. After that, we obtain conditions for $h_1, h_2$ that the corresponding subgroup $H$ is isolated. Then we study characters $\chi_1, \chi_2: H \rightarrow \CC^{*}$ which correspond to equivalent irreducible weight pairs $(H, \chi_1)$ and $(H, \chi_2)$. Finally, we obtain conditions for a character $\chi$ that $(H, \chi)$ is an irreducible weight pair.

Let us denote by $h_1, h_2$ two generators of $H / H \cap [G,G]$. Since $\mathrm{rk}_{2}(H) = 0$, it follows that the commutator $[h_1, h_2]$ is in the center of $G$. If $\chi(C)$ is a root of unity, then we can extend a weight pair $(H, \chi)$ to $(H', \chi')$ with the rank $\rk_2(H') = 2$. Hence if $\chi(C)$ is a root of unity, the pair $(H, \chi)$ is not irreducible. Then $\chi(C)$ is not a root of unity, and  $h_1$ and $h_2$ commute.
Then generators $h_1$ and $h_2$ have the following form:

$$
h_1 =
\begin{pmatrix}
     1 & a & b & 0 \\
     0 & 1 & 0 & e \\
     0 & 0 & 1 & 0 \\
     0 & 0 & 0 & 1
\end{pmatrix}, 
\quad
h_2 =
\begin{pmatrix}
     1 & 0 & b' & 0 \\
     0 & 1 & 0 & e' \\
     0 & 0 & 1 & f' \\
     0 & 0 & 0 & 1
\end{pmatrix}
$$
such that  $ae' + f'b = 0$. The last condition follows from the equality $[h_1, h_2]  = 1$.

Since $H$ is abelian, then by Lemma~\ref{prop:isol} it is isolated. 

The conditions for generators $h_1$ and $h_2$ that $H$ is isolated are easy to compute in this case. They are the following ones: $\mathrm{GCD}(a,b,e) = 1$ and $\mathrm{GCD}(f',b',e') = 1$.

\bigskip

Now for every $H \in \Sigma_{2, 0}$, we need to describe the fiber $Y_{2, 0; \, H}$.

For all subgroups $H \in \Sigma_{2, 0}$, the quotient $N_{G}(H) / H$ is generated by

$$
g_1 =
\begin{pmatrix}
     1 & 0 & 1 & 0 \\
     0 & 1 & 0 & 0 \\
     0 & 0 & 1 & 0 \\
     0 & 0 & 0 & 1
\end{pmatrix}, 
\quad
g_2 =
\begin{pmatrix}
     1 & 0 & 0 & 0 \\
     0 & 1 & 0 & 1 \\
     0 & 0 & 1 & 0 \\
     0 & 0 & 0 & 1
\end{pmatrix}
$$

The action of $g_1, g_2$ on a character $\chi$ is as follows:

\begin{equation}\label{eq:2}
\begin{split}
\chi^{g_1}(h_1) = \chi(h_1) \lambda^{f'}, \ \ \  \chi^{g_1}(h_2) = \chi(h_2) \\
\chi^{g_2}(h_1) = \chi(h_1),  \ \ \  \   \chi^{g_2}(h_2) = \chi(h_2) \lambda^{-a}.
\end{split}
\end{equation}

We can see that the action above is free if $\chi(C) = \lambda$ is not a root of unity, which means that corresponding representations are irreducible.

\end{proof}

Let us denote by $Z_{2, 0;\, H}^{(2, 1)} =  \{ \{ t \in \CC^*, \ \lambda \notin \mu_{\infty} \} / \sim\}$, by $ Z_{2, 0;\, H}^{(2, 2)} =  \{ \{ s \in \CC^{*}, \ \lambda \notin \mu_{\infty} \} / \sim \}  $ and by $Z_{2, 0;\, H}^{(1)} = \{  \lambda \notin \mu_{\infty} \} $, where $\sim$ is defined in~\ref{z-equiv}.

\medskip

\begin{corollary}\label{corollary:fib2}

If $H \in \Sigma_{2, 0}$, then the fiber $Z_{2, 0;\, H}$ of  $Z_{2, 0}$ over a subgroup $H$ has the following  iterated structure of a bundle, namely:
$$Z_{2, 0;\, H}^{(2, 1)}  \rightarrow Z_{2, 0;\, H}^{(1)} \ \ , \ \  Z_{2, 0;\, H}^{(2, 2)}  \rightarrow Z_{2, 0;\, H}^{(1)}.$$ We describe fibers of these bundles consecutively in coordinates $(t, s, \lambda)$.

$$ \{ E_{\lambda^{f'}} \} \ , \ \{ E_{\lambda^a} \} \ , \ \{\CC^* \setminus S^1\} \ \cup \ \{ P_{\lambda^{f'}} \} \ , \ \{ P_{\lambda^a} \} \ , \ \{S^1 \setminus \mu_{\infty} \}.$$

\end{corollary} 

\begin{proof}
Follows directly from $N_{G}(H) / H$ action on a character $\chi$ of $H$ above (equations ~\ref{eq:2}).
\end{proof}

\medskip

\begin{lemma}\label{lemma:fib2}
If $H_1 \in  \Sigma_{2, 0}$, then the following subgroups $H_2 \in \Sigma_{2, 0}$ are equivalent to $H_1$:

$H_2 = \phi((a, b - a \tilde{d}, e, f, b', e' + f' \tilde{d}))$ for an arbitrary integer $\tilde{d}$.

Fibers $Z_{2, 0; \, H_1}, Z_{2, 0;\, H_2} $ over subgroups $H_1, H_2$ can be canonically identified.
\end{lemma}

\begin{proof}
Since $H_1 =  H^{*}_{1}$, by Proposition~\ref{lem:isopairs} we need to study only such subgroups $H_2$ that there exists an element $g \in G$ such that~${(H_2)^g=H_1}$ and ${\chi_1|_{H_2^g\cap H_1} = \chi_2^g|_{H_2^g\cap H_1}}$. 
If $H_1 \in \Sigma_{2, 0}$, then the quotient $G / N_{G}(H_1)$ is generated by 
$$
g = 
\begin{pmatrix}
     1 & 0 & 0 & 0 \\
     0 & 1 & 1 & 0 \\
     0 & 0 & 1 & 0 \\
     0 & 0 & 0 & 1
\end{pmatrix}.
$$
Then it is easy to observe that conjugated subgroups are $\phi((~a,~b~-~a~\tilde{d},~e,~f',~b',~e'~+~f'~\tilde{d}))$. If we denote non-central generators of $H_2$ by $\tilde{h}_1, \tilde{h}_2$, then the characters of the subgroups are related as follows: $\chi_2(\tilde{h}_1) = \chi_{1}^{\tilde{d}}(h_1)$, $\chi_2(\tilde{h}_2) = \chi_{1}^{\tilde{d}}(h_2)$ and $\chi_2(C) = \chi_{1}(C)$. It gives a canonical identification of fibers $Z_{2, 0; \, H_1}, Z_{2, 0; \, H_2} $ over subgroups $H_1$ and~$H_2$.
\end{proof}

\section{ \texorpdfstring{$\text{The case of }\mathrm{rk}_{1}(H) = 2, \ \mathrm{rk}_{2}(H) = 1.$}{Case 3}}
\bigskip


Let us define sets $S_1, S_2$.
 $$S_1 =  \{ (a, e, d', e') \in \ZZ^{4}  \ \vert \ a \neq 0, \ d \neq 0 \  \text{and}  \  |\mathrm{GCD}(a, e) | = 1,  \ |\mathrm{GCD}(d', e') | = 1  \} .$$
$$ \hspace*{-1cm}
S_2 =  \{  (a, e, d', e') \in \ZZ^{4}   \ \vert \ a \neq 0, \ d \neq 0 \ \text{and} \  |\mathrm{GCD}(a, e) | = k_1,  \ |\mathrm{GCD}(d', e') | = k_2 , \ | k_1 k_2 | > 1  \}.$$

\medskip

\begin{lemma}\label{lemma:3}

There is a canonical bijection $\phi$ from $S_1 \ \cup \ S_2$ to $\Sigma_{2, 1}$. It maps a tuple $(a, e, d', e')$ to a subgroup $H$, generated by the following matrices $$
h_1 =
\begin{pmatrix}
     1 & a & 0 & 0 \\
     0 & 1 & 0 & e \\
     0 & 0 & 1 & 0 \\
     0 & 0 & 0 & 1
\end{pmatrix},
\quad
h_2 =
\begin{pmatrix}
     1 & 0 & 0 & 0 \\
     0 & 1 & d' & e' \\
     0 & 0 & 1 & 0 \\
     0 & 0 & 0 & 1
\end{pmatrix},
\quad
C =
\begin{pmatrix}
     1 & 0 & 0 & 1 \\
     0 & 1 & 0 & 0 \\
     0 & 0 & 1 & 0 \\
     0 & 0 & 0 & 1
\end{pmatrix}.
$$

If $H \in \phi(S_1)$, then we can extend the bijection $\phi$ to the map from $$(\CC^{*})^2 \times {\mu_N} \times \{ \CC^{*} \setminus \mu_{\infty} \} = \{ (t, r, z,  \lambda) \ \vert \  z^{ad'} 
\lambda^{ae'} = 1 \  \text{and} \ \lambda \notin \mu_{\infty}\} $$ to $Y_{2,1; \, H}$, which is 
defined as follows: $$ t = \chi(h_1) , \ \ r = \chi(h_2) , \ \ z = \chi(h_3), \ \  \lambda = \chi(C).$$

If $H \in \phi(S_2)$, we can extend the bijection $\phi$ to the map from  $(\CC^{*})^2 \times {\mu_N} \times \{ \CC^{*} \setminus \mu_{\infty} \} = \{ (t, r, z,  \lambda) \ \vert \  z^{ad'} \lambda^{ae'} ~= ~1, \ \text{the minimal natural m that} \ z^{\frac{a d' m}{k_1 k_2}} \lambda^{\frac{a e' m }{k_1 k_2}}  ~=~1 \\
\text{equals }|k_1 k_2|, \ \text{and} \ \lambda \notin \mu_{\infty}\} $ to $Y_{2,1; \, H}$, which is defined as above. 

\end{lemma}

\begin{proof}

The proof goes as follows. First, we prove that if $H \in \Sigma_{2, 1}$, then its generators may be chosen in the form given in Lemma~\ref{lemma:3}. Then we study characters $\chi_1, \chi_2: H \rightarrow \CC^{*}$ which correspond to equivalent irreducible weight pairs $(H, \chi_1)$ and $(H, \chi_2)$. After that, we obtain conditions for a character $\chi$ that $(H, \chi)$ is an irreducible weight pair. 

Let us denote the two generators of $H / H \cap [G,G]$ by $h_1, h_2$. If $\chi(C)$ is a root of unity, then we need to extend the weight pair $(H, \chi)$ to the pair $(H', \chi')$ with the rank $\rk_2(H') = 2$. Hence $\chi(C)$ is not a root of unity. Let us denote by $h_3$ the generator of $H \cap [G, G] / Z(G)$.
Let us first consider generators $h_1$ and $h_2$ in the following (general) form: 

$$
h_1 =
\begin{pmatrix}
     1 & a & b & 0 \\
     0 & 1 & d & e \\
     0 & 0 & 1 & f \\
     0 & 0 & 0 & 1
\end{pmatrix}, 
\quad
h_2 =
\begin{pmatrix}
     1 & a' & b' & 0 \\
     0 & 1 & d' & e' \\
     0 & 0 & 1 & f' \\
     0 & 0 & 0 & 1
\end{pmatrix}
$$

If $d = d' = 0$, then the commutator $[h_1, h_2]$ is in the center of the group $G$. But then $[h_1, h_3]$ and $[h_2, h_3]$ can not both be unity, hence $\chi(C)$ is a root of unity, which contradicts the earlier statement. Then either $d$ or $d'$ is not zero. Hence, the commutator $[h_1, h_2] = h_3^k C^n$ for some integers $k, n$. Then $\chi(h_3)^{ad} \chi(C)^{a e'} = 1$ (in particular, if $e' = 0$ we obtain that $\chi(h_3)$ is a root of unity). 
Besides, $[h_1,h_3] = [h_2, h_3] = 1$, otherwise $[h_1,h_3]$ and $ [h_2, h_3]$ are in the center of the group $G$, and $\chi(C)$ is a root of unity. Then the generator $h_3$ has to be proportional to $[h_1, h_2]$, that is, to the element

$$
\begin{pmatrix}
     1 & 0 & -a & 0 \\
     0 & 1 & 0 & f \\
     0 & 0 & 1 & 0 \\
     0 & 0 & 0 & 1
\end{pmatrix}.
$$

From  $[h_1,h_3] = [h_2, h_3] = 1$ we obtain that either $f = f' = 0$ or $a = a' = 0$. Let us consider the case of $f = f' = 0$ (the other one is treated similarly, if we put integer parameters $f = a, \ b = e, \ b' = e'$). Since $[h_1, h_2] = h_3^k C^n$, the generator $h_3$ has the following form:
 
$$
h_3 =
\begin{pmatrix}
     1 & 0 & b'' & 0 \\
     0 & 1 & 0 & 0 \\
     0 & 0 & 1 & 0 \\
     0 & 0 & 0 & 1
\end{pmatrix}.
$$
 
But the element

$$
\begin{pmatrix}
     1 & 0 & 1 & 0 \\
     0 & 1 & 0 & 0 \\
     0 & 0 & 1 & 0 \\
     0 & 0 & 0 & 1
\end{pmatrix}
$$
belongs to $C_{G}(H)$. Then by Remark~\ref{rem-cent}, it coincides with $h_3$. 

So we can have generators $h_1$ and $h_2$ in the following form (dividing $h_1$ and $h_2$ by $h_3^{b}$ and $h_3^{b'}$ correspondingly):

$$
h_1 =
\begin{pmatrix}
     1 & a & 0 & 0 \\
     0 & 1 & 0 & e \\
     0 & 0 & 1 & 0 \\
     0 & 0 & 0 & 1
\end{pmatrix}, 
\quad
h_2 =
\begin{pmatrix}
     1 & 0 & 0 & 0 \\
     0 & 1 & d' & e' \\
     0 & 0 & 1 & 0 \\
     0 & 0 & 0 & 1
\end{pmatrix}.
$$

A subgroup $H$ in this case is not necessarily isolated. Namely, it is not isolated if $|\mathrm{GCD}(a, e)|~>~1$ or $|\mathrm{GCD}(d, e')| > 1$.

\bigskip

Now for each $H \in \Sigma_{2,1}$ we describe the fiber $Y_{2,1; \, H}$.

For all subgroups $H \in \Sigma_{2, 1}$, the quotient $N_{G}(H) / H$ is generated by

$$
g =
\begin{pmatrix}
     1 & 0 & 0 & 0 \\
     0 & 1 & 0 & 1 \\
     0 & 0 & 1 & 0 \\
     0 & 0 & 0 & 1
\end{pmatrix}.
$$

The action of $g$ on a character $\chi$ is as follows:
\begin{equation}\label{eq:4}
\chi^{g}(h_1) = \chi(h_1) \lambda^{-a}, \ \ \  \chi^{g}(h_2) = \chi(h_2) , \ \ \  \chi^{g}(h_3) = \chi(h_3).
\end{equation}

If $ |\mathrm{GCD}(a, e) | = k_1,  \ |\mathrm{GCD}(d, e') | = k_2 , \ | k_1 k_2 | > 1$, and $ \chi(h_3)^{\frac{a d m}{k_1 k_2}} \chi(C)^{\frac{a e' m }{k_1 k_2}}  ~=~1$ for an integer $m$, then the elements $h_2^{\frac{m_1}{k_1}}$ and $h_2^{\frac{m_2}{k_2}}$ belong to $C_{G}(H) \setminus H$. It contradicts irreducibility of corresponding induced representation. 

Thus, we can see that given the conditions for a character $\chi$ of $H$ formulated in Lemma~\ref{lemma:3}, the action above~\ref{eq:4} is free. Then corresponding representations are irreducible.
\end{proof}

\medskip 

Let us denote by $Z_{2, 1;\, H}^{(2, 1)} =  \{ \{ (t, \lambda) \ \vert \ t \in \CC^*, \ \lambda \notin \mu_{\infty} \} / \sim\}$, by $Z_{2, 1;\, H}^{(1, 1)} =  \{ r \in \CC^* \}$,
by $Z_{2, 1;\, H}^{(2, 2)} =  \{ \{ (z, \lambda) \ \vert \ z \in \CC^*, \ \lambda \notin \mu_{\infty} \} / \sim\}$ and by $Z_{2, 1;\, H}^{(1, 2)} = \{ \lambda \notin \mu_{\infty} \}$, where $\sim$ is defined in~\ref{z-equiv}.

\medskip

\begin{corollary}\label{corollary:fib3}

If $H \in  \Sigma_{2, 1}$, then the fiber $Z_{2, 1;\, H}$ of  $Z_{2, 1}$ over a subgroup $H$ has iterated structure of a bundle, namely: $$Z_{2, 1;\, H}^{(2, 1)} \rightarrow Z_{2, 1;\, H}^{(1, 2)} \ , \ Z_{2, 1;\, H}^{(1, 1)}  \ , \ Z_{2, 1;\, H}^{(2, 2)} \rightarrow Z_{2, 1;\, H}^{(1,2)}. $$ We describe fibers of these bundles consecutively  in coordinates $(t, r, z, \lambda)$: 

$ \{ E_{\lambda^a} \} \ , \ \CC^{*} \ , \ \mu_{a d'} \ , \  \{\CC^* \setminus S^1\} \ \cup \ \{ P_{\lambda^a} \} \ , \  \CC^{*} \ , \ \mu_{a d'} \ , \  \{S^1 \setminus \mu_{\infty} \}$.

\end{corollary} 

\begin{proof}
Follows directly from $N_{G}(H) / H$ action on a character $\chi$ of a subgroup $H$ above (equations ~\ref{eq:4}).

We obtain the value of $z = (\lambda^{-ae'})^{\frac{1}{ad'}}$ from $z^{ad'} \lambda^{ae'} ~= ~1$.  
\end{proof}

\medskip

\begin{lemma}\label{lemma:fib3}

If $H_1 \in  \Sigma_{2, 1}$, then the following subgroups are equivalent to $H_1$:
$H_2 = \phi( (a, e, d', e' - d' \tilde{f}))$ for an arbitrary integer $\tilde{f}$.

If $H_1 \in  \phi(S_2)$, then there is also a finite set of subgroups which are $F$-equivalent to $H_1$.

Fibers $Z_{2, 1; \, H_1}$ and $Z_{2, 1; \, H_2} $ over equivalent subgroups may be identified canonically.
\end{lemma}

\begin{proof}

First, by Proposition~\ref{lem:isopairs}[(i)] we consider $G / N_{G}(H_{1}^{*})$ action on $H_1$. The quotient $G / N_{G}(H_{1}^{*})$ is generated by 
$$
g = 
\begin{pmatrix}
     1 & 0 & 0 & 0 \\
     0 & 1 & 0 & 0 \\
     0 & 0 & 1 & 1 \\
     0 & 0 & 0 & 1
\end{pmatrix}.
$$
It is easy to compute that conjugated subgroups are $\phi((a, e, d', e' - d' \tilde{f}))$. Let us denote generators of $H_2$ which generate $H_2 / \big(H_2 \cap [G, G] \big)$ by $\tilde{h}_1, \tilde{h}_2$, and we denote by $\tilde{h}_3$ the element which generates $\big( H_2 \cap [G, G] \big)/Z(G)$. Then characters of the subgroups are related as follows: $\chi_2(\tilde{h}_1) = \chi_{1}^{\tilde{f}}(h_1), \ \chi_2(\tilde{h}_2) = \chi_{1}^{\tilde{f}}(h_2), \ \chi_2(\tilde{h}_3) = \chi_{1}^{\tilde{f}}(h_3)$, and $\chi_2(C) = \chi_{1}(C)$. It gives a canonical identification of fibers $Z_{2, 1; \, H_1}$ and $Z_{2, 1; \, H_2} $ over subgroups $H_1$ and $H_2$.
\\
If $H_1 \in  \phi(S_2)$, then $H_1$ is not isolated. By Proposition~\ref{lem:isopairs}[(ii)] we need to consider also equivalent irreducible weight pairs $(H_1, \chi_1)$ and $(H_2, \chi_2)$ such that $H_1^{*} = H_{2}^{*} $ and $H_1 \neq H_2$. Clearly, we can only possibly extract roots from generators $h_1$ and $h_2$. The condition $z^{ad'} \lambda^{ae'} ~= ~1$ (which here stands for $\chi_1([h_1, h_2]) = 1$) holds as well for $H_2$, since $h_3 = \tilde{h}_3, \ \chi_1(h_3) = \chi_2(h_3) = z$ and $\chi_1(C) = \chi_2(C) = \lambda$. Let us denote $\mathrm{GCD}(a,e) = k_1$ and $\mathrm{GCD}(d',e') = k_2$. Then for every divisor $m_1$ of $k_1$, a subgroup $\phi((\frac{a}{m_1}, \frac{e}{m_1}, d' m_1, e' m_1))$ with a character $\chi_2$ defined by $\chi_1(h_1) = \chi_2(\tilde{h}_1)^{m_1}$, $\chi_1(h_2)^{m_1} = \chi(\tilde{h}_2)$ is $F$-equivalent to $H_1$. Similarly, a subgroup $\phi( ( a m_2, e m_2, \frac{d'}{m_2}, \frac{e'}{m_2}))$ with a character $\chi_2$ defined by $\chi_1(h_1)^{m_2} = \chi(\tilde{h}_1)$, $\chi_1(h_2) = \chi(\tilde{h}_2)^{m_2}$ is $F$-equivalent to $H_1$. 

\end{proof}

\section{ \texorpdfstring{$\text{The case of }\mathrm{rk}_{1}(H) = 1, \ \mathrm{rk}_{2}(H) = 2.$}{Case 4}}
\bigskip


Let us define sets $S_1, S_2, S_3$, and $A$.
$$
 S_1 =  \{ (a, d, f, b, e, b', e') \in \ZZ^{7} \ \vert \ a \neq 0, \ f \neq 0,  \ e' \neq 0, \ b' \neq 0, \ |b| < |b'|, \ |e| < |e'| \} ,
$$
$$
\hspace*{-0.5cm} S_2 =  \{ (a, d, f, b, e, b', e') \in \ZZ^{7} \vert \ a \neq 0, \ d \neq 0, \ f =  0,  \ e' \neq 0, \ b' = 1, \ b =0, \ |e| < |e'| \ \},$$
$$\hspace*{-0.5cm} S_3 = \{ (a, d, f, b, e, b', e') \in \ZZ^{7} \vert \ a = 0, \ d \neq 0, \ f \neq  0,  \ b' \neq 0, \ e' = 1, \ |b| < |b'|, \ e = 0 \}, $$
$$A = \{ (a, d, f, b, e, b', e') \in \ZZ^{7} \ \vert \ a = f = 0, \ d =1, \ b = 0, \ e = 0, \ b' = 1, \ e' = 1 \}.$$

Let $N = {\mathrm{GCD}(|ae'|, |fb'|)} $.

\medskip

\begin{lemma}\label{lemma:4}

There is a canonical bijection $\phi$ from $S_1 \ \cup \ S_2 \ \cup \ S_3 \ \cup \ A$ to $\Sigma_{1, 2}$. It maps a tuple $(a, d, f, b, e, b', e')$ to a subgroup $H$, generated by the following matrices $$
 \hspace*{-1cm}
 h_1 =
\begin{pmatrix}
     1 & a & b & 0 \\
     0 & 1 & d & e \\
     0 & 0 & 1 & f \\
     0 & 0 & 0 & 1
\end{pmatrix}, 
\quad
h_2 =
\begin{pmatrix}
     1 & 0 & b' & 0 \\
     0 & 1 & 0 & 0 \\
     0 & 0 & 1 & 0 \\
     0 & 0 & 0 & 1
\end{pmatrix},
\quad
h_3 =
\begin{pmatrix}
     1 & 0 & 0 & 0 \\
     0 & 1 & 0 & e' \\
     0 & 0 & 1 & 0 \\
     0 & 0 & 0 & 1
\end{pmatrix},
\quad
C =
\begin{pmatrix}
     1 & 0 & 0 & 1 \\
     0 & 1 & 0 & 0 \\
     0 & 0 & 1 & 0 \\
     0 & 0 & 0 & 1
\end{pmatrix}.
$$

If $H \in \phi(S_1) \ \cup \ \phi(S_2) \ \cup \ \phi(S_3)$, we can extend the bijection $\phi$ to the map from $$\CC^{*} \times  \{ \CC^{*} \setminus \mu_{\infty} \} \times  \{ \CC^{*} 
\setminus ~\mu_{\infty} \} \times ~\{\mu_{N} \} = \{ (t, z, w,  \lambda) \ \vert \  \ z, \ w \notin \mu_{\infty},\ \lambda^{N} = 1 \} $$ to $Y_{1,2; \, H}$, which is defined as 
follows: $$ t = \chi(h_1) , \ \ z = \chi(h_2) , \ \ w = \chi(h_3), \ \  \lambda = \chi(C).$$

If $H \in \phi(A)$, then $H$ is abelian (and there is only one such a subgroup). We can extend the bijection $\phi$ to the map from  $$ \CC^{*} \times  \{ \CC^{*} \setminus \mu_{\infty} \} \times  \{ \CC^{*} \setminus \mu_{\infty} \} \times \CC^{*}  = \{ (t, z, w,  \lambda) \ \vert \ z, \ w \notin \mu_{\infty} \} \ \cup $$ $$\  \CC^{*} \times \CC^* \times \CC^* \times  \{ \CC^{*} \setminus \mu_{\infty} \}  = \{ (t, z, w,  \lambda) \ \vert \ \lambda \notin \mu_{\infty} \}$$ to $Y_{1,2; \, H}$, which is defined as above. 

\end{lemma}

\begin{proof}

The proof goes as follows. We consider cases of non-abelian subgroups and the abelian subgroup separately. We study characters $\chi_1, \chi_2: H \rightarrow \CC^{*}$ which correspond to equivalent irreducible weight pairs $(H, \chi_1)$ and $(H, \chi_2)$. Then we obtain conditions for a character $\chi$ that $(H, \chi)$ is an irreducible weight pair, and we compute the fiber $Y_{1, 2; \, H}$ over a subgroup $H$.

Clearly, a subgroup $H$ with ranks $\mathrm{rk}_1(H) = 1, \ \mathrm{rk}_1(H) = 2$ can be generated as follows:
$$
h_1 =
\begin{pmatrix}
     1 & a & b & 0 \\
     0 & 1 & d & e \\
     0 & 0 & 1 & f \\
     0 & 0 & 0 & 1
\end{pmatrix}, 
\quad
h_2 =
\begin{pmatrix}
     1 & 0 & b' & 0 \\
     0 & 1 & 0 & 0 \\
     0 & 0 & 1 & 0 \\
     0 & 0 & 0 & 1
\end{pmatrix},
\quad
h_3 =
\begin{pmatrix}
     1 & 0 & 0 & 0 \\
     0 & 1 & 0 & e' \\
     0 & 0 & 1 & 0 \\
     0 & 0 & 0 & 1
\end{pmatrix}
$$

We have $\chi([h_1, h_2]) = \chi(C)^{f b'} = 1$ and $ \chi([h_1, h_3]) = \chi(C)^{a e'} =1$. From irreducibility criterion $S(H, \chi) = H$ we conclude that $b', e'$ have to be minimal natural numbers satisfying the condition $\chi(C)^{f b'} = \chi(C)^{a e'} =1$ (otherwise we can extend the weight pair). Obviously, integers $b'$ and $e'$ have to be non-zero and $|b|$, $|e|$ may be chosen to be smaller than $|b'|$, $|e'|$. If $f = 0$, then $b'=0$ from the condition $S(H, \chi) = H$, and hence $b = 0$ (similarly for $a = 0$ we have $e = 0$).

${\bf (i)}.$
First, let us consider the case of a subgroup $H \in \Sigma_{1, 2} \setminus \phi(A)$. Then either $a \neq 0$, or $f \neq 0$, and hence, $\chi(C)$ is a root of unity. Let us consider the first case with $a \neq 0$ (the other one with $ f \neq 0$ is treated similarly). If $d = 0$ then the weight pair may be extended to the one with ranks $\rk_1(H) = \rk_2(H) = 2$. Hence, $d$ is non-zero.

In this case, the quotient $N_{G}(H) / H$ is generated by

$$
g_1=
\begin{pmatrix}
     1 & 0 & 0 & 0 \\
     0 & 1 & d' & 0 \\
     0 & 0 & 1 & 0 \\
     0 & 0 & 0 & 1
\end{pmatrix}, 
\quad
g_2=
\begin{pmatrix}
     1 & 0 & 0 & 0 \\
     0 & 1 & 0 & 0 \\
     0 & 0 & 1 & f' \\
     0 & 0 & 0 & 1
\end{pmatrix},
$$
where $d'$ in $g_1$ is such a minimal natural number that $ad'$ is divisible by $b'$ and $fd'$ is divisible by $e'$, and $f'$ in $g_2$ is a minimal natural number that $f'd$ is divisible by $e'$. 

In this case a subgroup $H$ is not necessarily isolated.  

Generators $g_1, g_2$ act on a character $\chi$ as follows:
$$\chi^{g_1}(h_1) = \chi(h_1)  \chi(h_2)^{-\frac{ad'}{b'}}  \chi(h_3)^{\frac{fd'}{e'}},$$
$$\chi^{g_1}(h_2) = \chi(h_2), \ \ \  \chi^{g_1}(h_3) = \chi(h_3) $$
\begin{center}
and
\end{center}
$$\chi^{g_2}(h_1) = \chi(h_1)  \chi(h_3)^{-\frac{f'd}{e'}} \chi(C)^{-bf'}, $$
$$\chi^{g_2}(h_2) = \chi(h_2) \chi(C)^{-b'f'}, \ \ \  \chi^{g_2}(h_3) = \chi(h_3). $$
\medskip
From irreducibility criterion $S(H, \chi) = H$ we obtain that neither $\chi(h_2)$ nor $\chi(h_3)$ is a root of unity, otherwise we can extend the weight pair $(H, \chi)$ to the weight pair $(H', \chi')$ with $\mathrm{rk}_{1}(H') = \mathrm{rk}_{2}(H') = 2$. If these conditions are satisfied, then the action above is free, and corresponding representations are irreducible.

\bigskip
${\bf (ii)}.$
Now let us consider the case of the abelian subgroup $H \in \phi(A)$ (all the other cases were proven to correspond to non-abelian subgroups). Since $[h_1, h_2] = [h_1, h_3] = 1$, we have $a = f= 0$ for a generator $h_1$ in the following form:
$$
\begin{pmatrix}
     1 & a & b & 0 \\
     0 & 1 & d & e \\
     0 & 0 & 1 & f \\
     0 & 0 & 0 & 1
\end{pmatrix}.
 $$

From Lemma~\ref{prop:isol} it follows that since $H$ is abelian, it is isolated. Hence, generators of $H$ have the following form:

$$
h_1 =
\begin{pmatrix}
     1 & 0 & 0 & 0 \\
     0 & 1 & 1 & 0 \\
     0 & 0 & 1 & 0 \\
     0 & 0 & 0 & 1
\end{pmatrix}, 
\quad
h_2 =
\begin{pmatrix}
     1 & 0 & 1& 0 \\
     0 & 1 & 0 & 0 \\
     0 & 0 & 1 & 0 \\
     0 & 0 & 0 & 1
\end{pmatrix},
\quad
h_3 =
\begin{pmatrix}
     1 & 0 & 0 & 0 \\
     0 & 1 & 0 & 1 \\
     0 & 0 & 1 & 0 \\
     0 & 0 & 0 & 1
\end{pmatrix}
$$

Since $N_{G}(H) = G$, the quotient $N_{G}(H) / H$ is generated by

$$
g_1=
\begin{pmatrix}
     1 & 1 & 0 & 0 \\
     0 & 1 & 0 & 0 \\
     0 & 0 & 1 & 0 \\
     0 & 0 & 0 & 1
\end{pmatrix}, 
\quad
g_2=
\begin{pmatrix}
     1 & 0 & 0 & 0 \\
     0 & 1 & 0 & 0 \\
     0 & 0 & 1 & 1 \\
     0 & 0 & 0 & 1
\end{pmatrix}.
$$

Then $g_1$ and $g_2$ act on a character $\chi$ as follows:

$$\chi^{g_1}(h_1) = \chi(h_1)  \chi(h_2) , \ \ \  \chi^{g_1}(h_2) = \chi(h_2), \ \ \  \chi^{g_1}(h_3) = \chi(h_3) \lambda$$
\begin{center}
and
\end{center}
$$\chi^{g_2}(h_1) = \chi(h_1)  \chi(h_3)^{-1} , \ \ \  \chi^{g_2}(h_2) = \chi(h_2) \lambda^{-1}, \ \ \  \chi^{g_2}(h_3) = \chi(h_3). $$

\medskip
Then from irreducibility criterion $S(H, \chi) = H$ we obtain that $\chi(h_2)$, $ \chi(h_3)$ or $\chi(C)$ is not a root of unity. If these conditions are satisfied, then the action above is free, and corresponding representations are irreducible.

\end{proof}

\medskip 

Let us denote by $Z_{1, 2;\, H}^{(2)} =  \{ \{ (t, z, w,  \lambda) \ \vert \ t \in \CC^*, \ z, w \in \{ \CC^{*} \setminus \mu_{\infty} \}^2,\ \lambda^{N} = 1 \} / \sim\}$, by $ Z_{1, 2;\, H}^{(1)} =  \{  (z, w,  \lambda) \ \vert \  \ z, w \in \{ \CC^{*} \setminus \mu_{\infty} \}^2, \ \lambda^{N} = 1 \}  $, where $\sim$ is defined in~\ref{z-equiv}.

Let us denote by $Z_{1, 2, \, \text{ab};\, H}^{(2)} =  \{ \{ (t, z, w,  \lambda) \ \vert \ t \in \CC^*, \ z \ \text{or} \ w \ \text{or} \ \lambda \notin  \mu_{\infty} \} / \sim\}$, by $ Z_{1, 2, \, \text{ab};\, H}^{(1)} =  \{ ( z, w,  \lambda) \ \vert \ z \ \text{or} \ w \ \text{or} \ \lambda \notin  \mu_{\infty}  \}  $, where $\sim$ is defined in~\ref{z-equiv}.

\medskip

\begin{corollary}\label{corollary:fib4}

If $H \in \Sigma_{1, 2} \setminus \phi(A)$, then the fiber $Z_{1, 2;\, H}$ of  $Z_{1, 2}$ over a subgroup $H$ has the following  iterated structure of a bundle, namely:
$$Z_{1, 2;\, H}^{(2)}  \rightarrow Z_{1, 2;\, H}^{(1)}.$$ We describe fibers of this bundle consecutively in coordinates $(t, z, w, \lambda)$.

$$T_{z^{\frac{ad'}{b'}} w^{\frac{fd'}{e'}}, \, w^{\frac{f'd}{e'}} \lambda^{b f' }} \ , \ \{ \CC^{*} \setminus \mu_{\infty} \} \ , \ \{ \CC^{*} \setminus \ \mu_{\infty} \} \ , \ \{ 
\mu_{N} \};$$

If $H \in \phi(A)$, then the fiber $Z_{1, 2;\, H}$ of  $Z_{1, 2}$ over a subgroup $H$ has the following  iterated structure of a bundle, namely:
$$Z_{1, 2\, \text{ab};\, H}^{(2)}  \rightarrow Z_{1, 2\, \text{ab};\, H}^{(1)}.$$ We describe fibers of this bundle consecutively in coordinates $(t, z, w, \lambda)$.
$$T_{z, w} \ , \ \{ E_{\lambda}  \setminus \mu_{\infty} \} \ , \ \{ E_{\lambda} \setminus \mu_{\infty}  \} \ , \ \{ \CC^{*} \setminus S^1 \} \ \cup  $$
$$ T_{z, w} \ , \ \{ P_{\lambda}  \setminus \mu_{\infty} \} \ , \ \{ P_{\lambda} \setminus \mu_{\infty}  \} \ , \ \{S^1 \setminus \mu_{\infty} \} 
 \  \cup $$
$$  \{ E_{z} \} \ , \ \{ E_{\lambda} \setminus \mu_{\infty}  \} \ , \ \{ E_{\lambda} \cap  \mu_{\infty}  \} \ , \ \{ \CC^{*} \setminus  S^1 \} 
 \ \cup$$
$$   \{ E_{z} \} \ , \ \{ P_{\lambda} \setminus \mu_{\infty}  \} \ , \ \{ P_{\lambda} \cap  \mu_{\infty}  \} \ , \ \{ S^1 \setminus \mu_{\infty} \} \ \cup $$
$$\{ E_{w} \} \ , \ \{ E_{\lambda} \cap  \mu_{\infty}  \} \ , \ \{ E_{\lambda} \setminus \mu_{\infty}  \} \ , \  \{ \CC^{*} \setminus  S^1 \} \ \cup $$
$$\{ E_{w} \} \ , \ \{ P_{\lambda} \cap  \mu_{\infty}  \} \ , \ \{ P_{\lambda} \setminus \mu_{\infty}  \} \ , \  \{S^1 \setminus \mu_{\infty} \} \ \cup $$
$$ \{ \CC^{*} \} \ , \ \{ E_{\lambda} \cap  \mu_{\infty}  \} \ , \ \{ E_{\lambda} \cap \mu_{\infty} \} \ , \  \{ \CC^{*} \setminus  S^1 \} \ \cup$$
$$  \{ \CC^* \} \ , \ \{ P_{\lambda} \cap  \mu_{\infty}  \} \ , \ \{ P_{\lambda} \cap  \mu_{\infty} \} \ , \  \{ S^1 \setminus \mu_{\infty} \} \ \cup $$
$$ \{ T_{z,w} \} \ , \ \{  \CC^* \setminus  \mu_{\infty}  \} \ , \ \{  \CC^* \setminus  \mu_{\infty} \} \ , \  \{ \mu_{\infty} \} \ \cup $$
$$  \{ T_{z,w} \} \ , \ \{  \CC^* \setminus  \mu_{\infty}  \} \ , \ \{  \CC^* \setminus  \mu_{\infty} \} \ , \  \{ \mu_{\infty} \}.$$

\end{corollary} 
\medskip
\begin{proof}
Follows directly from $N_{G}(H) / H$ action on a character $\chi$ of $H$ above. 

\end{proof}

\medskip

\begin{lemma}\label{lemma:fib4}

If $H \in  \Sigma_{1, 2} \setminus \phi(A)$, then there are only subgroups, which are $F$-equivalent to $H$ (finitely many subgroups).

If $H \in \phi(A)$, then the corresponding subgroup is abelian and normal, and it is not equivalent to any other subgroups.

Fibers $Z_{1, 2; \, H_1}$ and $Z_{1, 2; \, H_2} $ over equivalent subgroups $H_1$ and $H_2$ may be identified canonically.

\end{lemma}

\begin{proof}

Let us consider $H_1 \in  \Sigma_{1, 2} \setminus \phi(A)$.

Since $N_{G}(H_1^{*}) = G$, by Proposition~\ref{lem:isopairs}[(ii)] we consider only equivalent irreducible weight pairs $(H_1, \chi_1)$ and $(H_2, \chi_2)$ with $H_1^{*} = H_{2}^{*} $ and $H_1 \neq H_2$. 


We can only possibly extract roots from generators $h_1$, $h_2$, and $h_3$. Then if there is an element $g_1 \in G \setminus H_1$ such that $g_1^{k_1} = h_1$, then by Proposition~\ref{lem:isopairs}[(ii)] for any divisor $m_1$ of $k_1$ there exists a finite set of exponents of generators $h_2$ and $h_3$ such that $S(H_2, \chi_2) = H_2$. Since $g_2^{b'} = h_2$ and $g_3^{e'} = h_3$, where 
$$
g_2=
\begin{pmatrix}
     1 & 0 & 1 & 0 \\
     0 & 1 & 0 & 0 \\
     0 & 0 & 1 & 0 \\
     0 & 0 & 0 & 1
\end{pmatrix}, 
\quad
g_3=
\begin{pmatrix}
     1 & 0 & 0 & 0 \\
     0 & 1 & 0 & 1 \\
     0 & 0 & 1 & 0 \\
     0 & 0 & 0 & 1
\end{pmatrix},
$$
then for any divisor $m_2$ of $b'$ and any divisor $m_3$ of $e'$, there also exists a finite set of corresponding exponents. We omit concrete expressions of parameters of equivalent weight pairs for their cumbersomeness (but they are easy to compute).

Since $a, d, f, b', e'$ are fixed, once we fixed the exponents of generators of $H_1$, there are only parameters $b < |b'|$ and $e < |e'|$ , which yet can produce equivalent irreducible weight pairs. Let us consider $H_2 = \phi((a, d, f, 0, 0, b', e'))$ with a character $\chi_2$ defined by $\chi_2(\tilde{h}_1) = \chi_1(h_1) \chi_1(h_2)^{-\frac{b}{b'}}  \chi_1(h_3)^{-\frac{e}{e'}}$ and $\chi_1(\big< h_2, h_3, C \big >) =\chi_2(\big< h_2, h_3, C \big >)$. Then $H_1^{*} = H_2^{*}$ and $\chi_1 \vert_{H_1 \cap H_2} =\chi_2 \vert_{H_1 \cap H_2} $. All such irreducible weight pairs are $F$-equivalent.

If $H \in  \phi(A)$, then the corresponding subgroup $H$ is normal and abelian. Hence, $H$ is isolated, and by Proposition~\ref{lem:isopairs} the weight pair $(H, \chi)$ can not be equivalent to any irreducible weight pair with a different subgroup $H_2$.

\end{proof}

\section{ \texorpdfstring{$\text{The case of }\mathrm{rk}_{1}(H) = 2, \ \mathrm{rk}_{2}(H) = 2.$}{Case 5}}
\bigskip


Let  us define sets $S_1, S_2, S_3$ and $S_4$.
$$\hspace*{-0.3cm} S_1 =  \{ (a, f, b, e, d', f', b', e', b'', e'') \in \ZZ^{10} \ \vert \ a \neq 0, \ d \neq 0, \ f' \neq 0, \ f \neq 0,  \ b'' \neq 0, \ e'' \neq 0, $$ $$\ d' f \divby e'', \ d' a \divby b'', \ |b| < |b''|, \ |e| < |e''|, \  |b'| < |b''|, \ |e'| < |e''| \}, $$
$$S_2 =  \{ (a, f, b, e, d', f', b', e', b'', e'') \in \ZZ^{10} \ \vert \ a \neq 0, \ d \neq 0, \ f = f' = 0, \ e'' \neq 0, \ b'' = 1, $$ $$\ b = b' = 0, \ |e| < |e''|, \ |e'| < |e''| \},$$ 
$$S_3 = \{ (a, f, b, e, d', f', b', e', b'', e'') \in \ZZ^{10} \ \vert \ a = 0, \ d \neq 0, \ f' \neq 0, \ f = 0, \ e'' = 1, \ b'' \neq 0, $$ $$ \ e = e' = 0, \  |b'| < |b''|, \ |b| < |b''| \},$$
$$S_4 =  \{ ( a, f, b, e, d', f', b', e', b'', e'')  \in \ZZ^{10} \ \vert \ a \neq 0, \ f' \neq 0, \ d' = 0,  \ b'' \neq 0,$$ $$ \ e'' \neq 0,  \ |b| < |b''|, \ |e| < |e''|, \  |b'| < |b''|, \ |e'| < |e''|  \}.$$
\medskip

Let $N = \mathrm{GCD}(f b'', f' b'', a e'')$. Let us denote by $C_{z,w}$ the following curve: $\{ w^{\frac{-d' f}{e''}} z^{\frac{-d' a}{b''}}  \lambda^{ a e' + b f' - b'f } = 1 \ \vert \  z, w \in (\CC^{*})^2, \ z \  \text{or} \ w \notin \mu_{\infty}, \ \lambda^N = 1 \}$. Let us denote by $C_{z,w, sing}$ the following manifold: $\{ w, z \in C_{z,w} \ \text{such that} \ z \in \mu_{\infty} \ \text{and} \ w \in \{ S^1 \setminus \mu_{\infty} \} \, \}$. 
Let $N_2$ be a minimal natural number such that if $z^{\frac{-d' a}{b''}}  \lambda^{ a e'} = 1$, then $z^{N_2} = 1$.
Let $N_3$ be a minimal natural number such that if $w^{\frac{-d' f}{e''}} \lambda^{b f' - b'f }= 1$, then $w^{N_3} = 1$.

\medskip

\begin{lemma}\label{lemma:5}

There is a canonical bijection $\phi$ from $S_1 \ \cup \ S_2 \ \cup \ S_3 \ \cup \ S_4$ to $\Sigma_{2, 2}$. It maps a tuple $(a, f, b, e, d', f', b', e', b'', e'')$ to a subgroup $H$, generated by the following matrices $$
h_1 =
\begin{pmatrix}
     1 & a & b & 0 \\
     0 & 1 & 0 & e \\
     0 & 0 & 1 & f \\
     0 & 0 & 0 & 1
\end{pmatrix}, 
\quad
h_2 =
\begin{pmatrix}
     1 & 0 & b' & 0 \\
     0 & 1 & d' & e' \\
     0 & 0 & 1 & f' \\
     0 & 0 & 0 & 1
\end{pmatrix},
$$

$$
h_3 =
\begin{pmatrix}
     1 & 0 & b'' & 0 \\
     0 & 1 & 0 & 0 \\
     0 & 0 & 1 & 0 \\
     0 & 0 & 0 & 1
\end{pmatrix},
\quad
h_4 =
\begin{pmatrix}
     1 & 0 & 0 & 0 \\
     0 & 1 & 0 & e'' \\
     0 & 0 & 1 & 0 \\
     0 & 0 & 0 & 1
\end{pmatrix},
\quad
C =
\begin{pmatrix}
     1 & 0 & 0 & 1 \\
     0 & 1 & 0 & 0 \\
     0 & 0 & 1 & 0 \\
     0 & 0 & 0 & 1
\end{pmatrix}.
$$

If $H \in \phi(S_1)$, we can extend the bijection $\phi$ to the map from $$ \CC^{*} \times \CC^{*} \times C_{z, w} \times \mu_{N}  = \{ (t, s, z, w,  \lambda) \ \vert \ w^{\frac{-d' f}{e''}} z^{\frac{-d' a}{b''}}  \lambda^{ a e' + b f' - b'f } = 1, \ \lambda \in \mu_{N} \}$$ to $Y_{2,2; \, H}$, which is 
defined as follows: $$ t = \chi(h_1) , \ \ s = \chi(h_2) , \ \ z = \chi(h_3), \ \ w = \chi(h_4), \ \  \lambda = \chi(C).$$

If $H \in \phi(S_2)$, we can extend the bijection $\phi$ to the map from $$ \CC^{*} \times \CC^{*} \times \mu_{N_2} \times\{ \CC^{*} \setminus \mu_{\infty} \} \times \mu_{a e''}  = \{ (t, s, z, w,  \lambda) \ \vert  \ z^{\frac{-d' a}{b''}}  \lambda^{ a e'} = 1 , \ \lambda \in \mu_{a e''} \}$$ to $Y_{2,2; \, H}$, which is defined as above.

If $H \in \phi(S_3)$, we can extend the bijection $\phi$ to the map from $$ \CC^{*} \times \CC^{*} \times\{ \CC^{*} \setminus \mu_{\infty} \} \times \mu_{N_3} \times \mu_{f' b''}  = \{ (t, s, z, w,  \lambda) \ \vert \ w^{\frac{-d' f}{e''}} \lambda^{b f' - b'f } = 1, \ \lambda \in \mu_{f' b''} \}$$ to $Y_{2,2; \, H}$, which is defined as above.

If $H \in \phi(S_4)$, we can extend the bijection $\phi$ to the map from $$ \CC^{*} \times \CC^{*} \times \{ \CC^{*} \setminus \mu_{\infty} \} \times \CC^{*} \times \mu_{N}  = \{ (t, s, z, w,  \lambda) \ \vert \ z \ \text{or} \ w \notin \mu_{\infty}, \ \lambda \in \mu_{N} \}$$ to $Y_{2,2; \, H}$, which is defined as above.

\end{lemma}

\begin{proof}

The proof goes as follows. We consider separately cases of $H~\in~\big( \phi(S_1) \ \cup \ \phi(S_2) \ \cup \ \phi(S_3) \big)$ and the case of $H \in  \phi(S_4)$. Then we study characters $\chi_1, \chi_2:~H~\rightarrow~\CC^{*}$ which correspond to equivalent irreducible weight pairs $(H, \chi_1)$ and $ (H, \chi_2)$. Then we obtain conditions for a character $\chi$ that $(H, \chi)$ is an irreducible weight pair, and we compute the fiber $Y_{2, 2; \, H}$.

\medskip
$(1).$  If $H \in \phi(S_1)$, $\phi(S_2)$ or $\phi(S_3)$, then all cases of such subgroups are treated similarly, and without loss of generality we consider the case of a subgroup $H \in \phi(S_1)$. 

If $H \in \phi(S_1)$, then a subgroup $H$ can be generated as follows:

$$
h_1 =
\begin{pmatrix}
     1 & a & b & 0 \\
     0 & 1 & 0 & e \\
     0 & 0 & 1 & f \\
     0 & 0 & 0 & 1
\end{pmatrix}, 
\quad
h_2 =
\begin{pmatrix}
     1 & 0 & b' & 0 \\
     0 & 1 & d' & e' \\
     0 & 0 & 1 & f' \\
     0 & 0 & 0 & 1
\end{pmatrix},
$$

$$
h_3 =
\begin{pmatrix}
     1 & 0 & b'' & 0 \\
     0 & 1 & 0 & 0 \\
     0 & 0 & 1 & 0 \\
     0 & 0 & 0 & 1
\end{pmatrix},
\quad
h_4 =
\begin{pmatrix}
     1 & 0 & 0 & 0 \\
     0 & 1 & 0 & e'' \\
     0 & 0 & 1 & 0 \\
     0 & 0 & 0 & 1
\end{pmatrix}
$$

Since $[h_1, h_2] = h_3^{n_1} h_4^{n_2} C^{n_3}$ for some integers $n_1, n_2, n_3$, we obtain that $\ d' f~\divby~e''$ and $d' a~\divby~b'$. Clearly we can generate $H$ with parameters of generators satisfying $\ |b| < |b''|, \ |e| < |e''|, \  |b'| < |b''|, \ |e'| < |e''|$.

We have
\begin{equation}
\chi([h_1, h_2]) =  w^{\frac{-d' f}{e''}} z^{\frac{-d' a}{b''}}  \lambda^{-a d' f - ad' f' + a e' + b f' - b'f } = 1,
\end{equation}
$$\chi([h_1, h_3]) =  \lambda^{-f b''}= 1, \ \ \ \chi([h_1, h_4]) =  \lambda^{a e''}= 1,$$
$$\chi([h_2, h_3]) =  \lambda^{-f' b''}= 1.$$

Since $\ d' f \divby~e'', \ d' a~\divby b'$ and $\lambda^{\mathrm{GCD}(f, f') b''} = \lambda^{a e''} =1$, we have $ \lambda^{-a d' f - ad' f' } = 1$.

Then 
\begin{equation}\label{eq:comm}
\chi([h_1, h_2]) =  w^{\frac{-d' f}{e''}} z^{\frac{-d' a}{b''}}  \lambda^{-a d' f - ad' f' + a e' + b f' - b'f } = w^{\frac{-d' f}{e''}} z^{\frac{-d' a}{b''}}  \lambda^{ a e' + b f' - b'f } = 1.
\end{equation}

The quotient $N_{G}(H) / H$ is generated by

$$
g =
\begin{pmatrix}
     1 & 0 & 0 & 0 \\
     0 & 1 & 0 & 0 \\
     0 & 0 & 1 & \tilde{f} \\
     0 & 0 & 0 & 1
\end{pmatrix}
$$
where $\tilde{f}$ is the minimal natural number such that $\tilde{f} d'$ is divisible by $e''$.

The generator $g$ acts on a character $\chi$ as follows:
$$\chi^{g}(h_1) = \chi(h_1) \lambda^{-b \tilde{f}}, \ \ \  \chi^{g}(h_2) = \chi(h_2) \chi(h_4)^{-\frac{d'\tilde{f}}{e''}} \lambda^{-b' \tilde{f}}$$
\begin{center}
and
\end{center}
$$\chi^{g}(h_3) = \chi(h_3)  \lambda^{-b''\tilde{f}} , \ \ \  \chi^{g}(h_4) = \chi(h_4). $$

\medskip

Then from irreducibility criterion $S(H, \chi) = H$ we obtain that the central character $\chi(C)$ is a root of unity of order $N$, where $N = \mathrm{GCD}(f'b'', a e'', f b'')~=~1$.

From the equation~\ref{eq:comm}, we obtain conditions for $w$ and $z$ defining $N_2$ and $N_3$. If both $z$ and $w$ are roots of unity, then the weight pair extends to the one with $\rk_1(H) = 3$. Then for all three cases of $\phi(S_1), \phi(S_2)$ and $\phi(S_3)$ we have $Z_{2,2; \, H} \subseteq \{ \CC^{*} \times  \CC^{*} \times \mu_{\infty} \times\{ \CC^{*} \setminus \mu_{\infty} \} \times \mu_{\infty} \}$ (concrete formulas are in Lemma~\ref{lemma:4} formulation).

\bigskip
${\bf (ii)}.$  If $H \in  \phi(S_4)$, then a subgroup $H$ may be generated as follows:

$$
h_1 =
\begin{pmatrix}
     1 & a & b & 0 \\
     0 & 1 & 0 & e \\
     0 & 0 & 1 & 0 \\
     0 & 0 & 0 & 1
\end{pmatrix}, 
\quad
h_2 =
\begin{pmatrix}
     1 & 0 & b' & 0 \\
     0 & 1 & 0 & e' \\
     0 & 0 & 1 & f' \\
     0 & 0 & 0 & 1
\end{pmatrix},
$$

$$
h_3 =
\begin{pmatrix}
     1 & 0 & b'' & 0 \\
     0 & 1 & 0 & 0 \\
     0 & 0 & 1 & 0 \\
     0 & 0 & 0 & 1
\end{pmatrix},
\quad
h_4 =
\begin{pmatrix}
     1 & 0 & 0 & 0 \\
     0 & 1 & 0 & e'' \\
     0 & 0 & 1 & 0 \\
     0 & 0 & 0 & 1
\end{pmatrix}
$$

Then we have
$$\chi([h_1, h_2]) =   \lambda^{a e' + b f' } = 1,$$

$$\chi([h_1, h_4]) =  \lambda^{a e''}= 1,$$

$$\chi([h_2, h_3]) =  \lambda^{-f' b''}= 1.$$

The quotient $N_{G}(H) / H$ is generated by

$$
g =
\begin{pmatrix}
     1 & 0 & 0 & 0 \\
     0 & 1 & \tilde{d} & 0 \\
     0 & 0 & 1 & 1 \\
     0 & 0 & 0 & 1
\end{pmatrix},
$$
where $\tilde{d}$ is the minimal natural number such that $\tilde{d} a$ is divisible by $b''$ and $\tilde{d} f'$ is divisible by $e''$.

The generator $g$ acts on a character $\chi$ as follows:

$$\chi^{g}(h_1) = \chi(h_1) \chi(h_3)^{-\frac{a \tilde{d}}{b''}} , \ \ \  \chi^{g}(h_2) = \chi(h_2) \chi(h_4)^{-\frac{f' \tilde{d}}{e''}} $$
\begin{center}
and
\end{center}
$$\chi^{g}(h_3) = \chi(h_3) , \ \ \  \chi^{g}(h_4) = \chi(h_4). $$

\medskip

Then from irreducibility criterion $S(H, \chi) = H$ we obtain that the central character $\chi(C)$ is a root of unity of order $N$, which was defined earlier.

If both $z$ and $w$ are roots of unity, then the weight pair extends to the one with $\rk_1(H) = 3$. Thus, if $H \in \phi(S_4)$, then $Z_{2,2; \, H}$ is bijective to $ \{ \CC^{*}~\times~\CC^{*}~\times~\{ \CC^{*}~\setminus \mu_{\infty} \} \times \CC^{*} \times \mu_{N} \}  = \{ (t, s, z, w,  \lambda) \ \vert \ z \ \text{or} \ w \notin \mu_{\infty}, \ \lambda \in \mu_{N} \} $. 

\end{proof}

Let us denote by $Z_{2, 2;\, H}^{(2)} =  \{ \{ (t, s, z, w, \lambda) \ \vert \ t, \ s \in (\CC^*)^2, \  \vert \ w^{\frac{-d' f}{e''}} z^{\frac{-d' a}{b''}}  \lambda^{ a e' + b f' - b'f } = 1, \ \lambda \in \mu_{N}  \} / \sim\}$ and by $Z_{2, 2;\, H}^{(1)} =  \{  (z, w, \lambda) \ \vert w^{\frac{-d' f}{e''}} z^{\frac{-d' a}{b''}}  \lambda^{ a e' + b f' - b'f } = 1, \ \lambda \in \mu_{N}  \}$, where $\sim$ is defined in~\ref{z-equiv}.

\medskip 

\begin{corollary}\label{corollary:fib5}

If $H \in \phi(S_1)$, then the fiber $Z_{2, 2;\, H}$ of  $Z_{2, 2}$ over a subgroup $H$ has iterated structure of a bundle, namely: $$Z_{2, 2;\, H}^{(2)} \rightarrow Z_{2, 2;\, H}^{(1)}. $$ We describe fibers of this bundle consecutively in coordinates $(t, s, z, w, \lambda)$: 

$$  \CC^{*} \ , \ E_{w^{\frac{d'\tilde{f}}{e''}} } \ , \ C_{z, w} \ , \ \mu_{N}  \ \cup \  \CC^{*} \ , \ P_{w^{\frac{d'\tilde{f}}{e''}} }  \ , \ C_{z, w, sing} \ , \ \mu_{N};$$ 

If $H \in \phi(S_2)$, then the fibers $Z_{2, 2;\, H}^{(2)} \rightarrow Z_{2, 2;\, H}^{(1)}$ over $H$ are canonically bijective~to: 
 $$  \CC^{*} \ , \ E_{w^{\frac{d'\tilde{f}}{e''}} } \ , \ \mu_{N_2} \ , \ \{ \CC^* \setminus S^1\} \ , \ \mu_{a e''} \ \cup \  \CC^{*} \ , \ P_{w^{\frac{d'\tilde{f}}{e''}} }  \ , \ \mu_{N_2} \ , \ \{S^1 \setminus \mu_{\infty} \} \ , \ \mu_{a e''} ;$$

If $H \in \phi(S_3)$, then the fibers $Z_{2, 2;\, H}^{(2)} \rightarrow Z_{2, 2;\, H}^{(1)}$ over $H$ are canonically bijective~to: 
 $$  E_{z^{\frac{d' a}{b''}}} \ , \ \CC^* \ , \ \{ \CC^* \setminus S^1 \}  \ , \ \mu_{N_3} \ , \ \mu_{f' b''} \ \cup \  P_{z^{\frac{d' a}{b''}}} \ , \ \CC^* \ , \ \{S^1 \setminus \mu_{\infty} \}  \ , \ \mu_{N_3} \ , \ \mu_{f' b''};$$

If $H \in \phi(S_4)$, then the fibers $Z_{2, 2;\, H}^{(2)} \rightarrow Z_{2, 2;\, H}^{(1)}$ over $H$ are canonically bijective~to: 
$$ E_{z^{\frac{d' a}{b''}}}  \ , \ E_{w^{\frac{d'\tilde{f}}{e''}} } \ , \ \{ \CC^{*} \setminus S^1 \} \ , \  \{ \CC^* \setminus S^1 \} \ , \ \mu_{N} \ \cup \  P_{z^{\frac{d' a}{b''}}}  \ , \ E_{w^{\frac{d'\tilde{f}}{e''}} } \ , \ \{S^1 \setminus \mu_{\infty} \} \ , \ \{ \CC^* \setminus S^1 \} \ , \ \mu_{N} \ \cup \ $$
$$ E_{z^{\frac{d' a}{b''}}}  \ , \ P_{w^{\frac{d'\tilde{f}}{e''}} } \ , \ \{ \CC^* \setminus S^1 \}  \ , \ \{S^1 \setminus \mu_{\infty} \} \ , \ \mu_{N} \ \cup \ P_{z^{\frac{d' a}{b''}}}  \ , \ P_{w^{\frac{d'\tilde{f}}{e''}} }  \ , \  \{S^1 \setminus \mu_{\infty} \}  \ , \  \{S^1 \setminus \mu_{\infty} \}  \ , \ \mu_{N} \ \cup \ $$
$$\CC^{*}  \ , \  E_{w^{\frac{d'\tilde{f}}{e''}} }  \ , \ \mu_{\infty}  \ , \  \{ \CC^* \setminus S^1 \}  \ , \ \mu_{N} \ \cup \  E_{z^{\frac{d' a}{b''}}}   \ , \ \CC^{*}  \ , \   \{ \CC^* \setminus S^1 \} \ , \ \mu_{\infty}  \ , \ \mu_{N} \ \cup \  $$
$$\CC^{*}  \ , \ P_{w^{\frac{d'\tilde{f}}{e''}} }  \ , \ \mu_{\infty}  \ , \  \{S^1 \setminus \mu_{\infty} \}  \ , \ \mu_{N} \ \cup \ P_{z^{\frac{d' a}{b''}}}   \ , \ \CC^{*}  \ , \  \{S^1 \setminus \mu_{\infty} \}  \ , \ \mu_{\infty}  \ , \ \mu_{N} .$$

\end{corollary} 

\begin{proof}
Follows directly from $N_{G}(H) / H$ action on a character $\chi$ of $H$ above. 

\end{proof}

\medskip

\begin{lemma}\label{lemma:fib5}

If $H \in  \Sigma_{2, 2}$, then there are only subgroups, which are $F$-equivalent to $H$ (finitely many subgroups).

Fibers $Z_{2, 2; \, H_1}$ and $Z_{2, 2;\, H_2} $ over equivalent subgroups $H_1$ and $H_2$ may be identified canonically.
\end{lemma}

\begin{proof}

Since $N_{G}(H_1^{*}) = G$, by Proposition~\ref{lem:isopairs}[(ii)] we need to consider only equivalent irreducible weight pairs $(H_2, \chi_2)$ such that $H_1^{*} = H_{2}^{*} $ and $H_1 \neq H_2$. 

Values of a character $\chi_1$ in this case have to satisfy the following conditions:  $$\lambda^{f' b''} = \lambda^{f b''} = \lambda^{a e''} = 1,$$ $$ 
w^{\frac{-d' f}{e''}} z^{\frac{-d' a}{b''}}  \lambda^{ a e' + b f' - b'f } = 1.$$ Let us denote $\mathrm{GCD}(f' b'', f b'', a e'')$ by $N$. 

We can only possibly extract roots from generators $h_1$, $h_2$, $h_3$, and $h_4$. If there is $g_1 \in G \setminus H_1$ such that $g_1^{k_1} = h_1$, or  
$g_2 \in G \setminus H_1$ such that $g_2^{k_2} = h_2$, then by Proposition~\ref{lem:isopairs}[(ii)] for any divisor $m_1$ of $k_1$ or $m_2$ of $k_2$, there exists a finite set of exponents of other generators such that $S(H_2, \chi_2) = H_2$. Similarly, for any divisor $m_3$ of $b'$ and any divisor $m_4$ of $e'$, there also exists a finite set of corresponding exponents. We omit concrete expressions of parameters of equivalent irreducible weight pairs due to their cumbersomeness.

Since $a, f, d', f', b'', e''$ are fixed, once we fixed the exponents of generators of $H_1$, there are only parameters $b, b' < |b''|$ and $e, e' < |e''|$ which can be varied to produce equivalent irreducible weight pairs. Since we have $$ 
w^{\frac{-d' f}{e''}} z^{\frac{-d' a}{b''}}  \lambda^{ a e' + b f' - b'f } = 1,$$
the expression  $a e' + b f' - b'f$ can not be changed modulo $N$, since $\lambda^{N} = 1$. Thus, if we replace  $b, b' < |b''|$ and $e, e' < |e''|$ by  $\tilde{b}, \tilde{b'} < |b''|$ and $\tilde{e}, \tilde{e'} < |e''|$ such that $$(a e' + b f' - b'f) \  \equiv \ (a \tilde{e'} +\tilde{b} f' - \tilde{b'} f) (\mathrm{mod} \, N),$$ then there exists a corresponding character $\chi_2$ of a corresponding subgroup $H_2$ such that $\chi_1 \vert_{H_1 \cap H_2} =\chi_2 \vert_{H_1 \cap H_2} $. All such weight pairs are also $F$-equivalent to the weight pair $(H_1, \chi_1)$ .

\end{proof}

\section{ \texorpdfstring{$\text{The case of }\mathrm{rk}_{1}(H) = 3, \ \mathrm{rk}_{2}(H) = 2.$}{Case 6}}
\bigskip


Let $S =  \{ (a, b, e, d', b', e', f'', b'', e'', b''', e''') \in \ZZ^{11} \ \vert \ a \neq 0, \ d' \neq 0, \ f'' \neq 0, \ b''' \neq 0, \ e''' \neq 0, \ |b| < |b'''|, \ |e| < |e'''|, \  |b'| < |b'''|, \ |e'| < |e'''|,  \ |b''| < |b'''|, \ |e''| < |e'''| \} $.

Let us denote $\mathrm{GCD}(f'' b''', a e'' + b f'', a e''')$ by $N_3$, and let $\lambda^{N_3} = 1$. Let $N_1$ be a minimal natural number such that $z^{N_1} = 1$, if $z^{\frac{-d' a}{b'''}}  \lambda^{ -a e'} = 1$. Let $N_2$ be a minimal natural number such that $w^{N_2} = 1$, if $w^{\frac{d' f''}{e'''}}  \lambda^{ b' f''} = 1$. 

\medskip

\begin{lemma}\label{lemma:6}

There is a canonical bijection $\phi$ from $S$ to $\Sigma_{3, 2}$. It maps a tuple $(a, b, e, d', b', e', f'', b'', e'', b''', e''') $ to a subgroup $H$, generated by the following matrices $$
h_1 =
\begin{pmatrix}
     1 & a & b & 0 \\
     0 & 1 & 0 & e \\
     0 & 0 & 1 & 0 \\
     0 & 0 & 0 & 1
\end{pmatrix}, 
\quad
h_2 =
\begin{pmatrix}
     1 & 0 & b' & 0 \\
     0 & 1 & d' & e' \\
     0 & 0 & 1 & 0 \\
     0 & 0 & 0 & 1
\end{pmatrix},
\quad
h_3 =
\begin{pmatrix}
     1 & 0 & b'' & 0 \\
     0 & 1 & 0 & e'' \\
     0 & 0 & 1 & f'' \\
     0 & 0 & 0 & 1
\end{pmatrix},
$$

$$
h_4 =
\begin{pmatrix}
     1 & 0 & b''' & 0 \\
     0 & 1 & 0 & 0 \\
     0 & 0 & 1 & 0 \\
     0 & 0 & 0 & 1
\end{pmatrix},
\quad
h_5 =
\begin{pmatrix}
     1 & 0 & 0 & 0 \\
     0 & 1 & 0 & e''' \\
     0 & 0 & 1 & 0 \\
     0 & 0 & 0 & 1
\end{pmatrix},
\quad
C =
\begin{pmatrix}
     1 & 0 & 0 & 1 \\
     0 & 1 & 0 & 0 \\
     0 & 0 & 1 & 0 \\
     0 & 0 & 0 & 1
\end{pmatrix}.
$$

If $H \in \phi(S)$, we can extend the bijection $\phi$ to the map from  $$ \CC^{*} \times \CC^{*} \times  \CC^{*} \times \mu_{N_1}  \times  \mu_{N_2} \times  \mu_{N_3} = \{ (t, r, s, z, w,  \lambda) \ \vert \ z \in \mu_{N_1}, \ w \in \mu_{N_2}, \ \lambda \in \mu_{N_3} \}$$ to $Y_{3,2; \, H}$, which is 
defined as follows: $$ t = \chi(h_1) , \ \ r = \chi(h_2) , \ \ s = \chi(h_3), \ \ z = \chi(h_4), \ \ w = \chi(h_5), \ \  \lambda = \chi(C).$$
\end{lemma}

\begin{proof}

If $H \in \Sigma_{3, 2}$, then $H$ may be generated as follows:

$$
h_1 =
\begin{pmatrix}
     1 & a & b & 0 \\
     0 & 1 & 0 & e \\
     0 & 0 & 1 & 0 \\
     0 & 0 & 0 & 1
\end{pmatrix}, 
\quad
h_2 =
\begin{pmatrix}
     1 & 0 & b' & 0 \\
     0 & 1 & d & e' \\
     0 & 0 & 1 & 0 \\
     0 & 0 & 0 & 1
\end{pmatrix},
\quad
h_3 =
\begin{pmatrix}
     1 & 0 & b'' & 0 \\
     0 & 1 & 0 & e'' \\
     0 & 0 & 1 & f \\
     0 & 0 & 0 & 1
\end{pmatrix},
$$
$$
h_4 =
\begin{pmatrix}
     1 & 0 & b''' & 0 \\
     0 & 1 & 0 & 0 \\
     0 & 0 & 1 & 0 \\
     0 & 0 & 0 & 1
\end{pmatrix},
\quad
h_5 =
\begin{pmatrix}
     1 & 0 & 0 & 0 \\
     0 & 1 & 0 & e''' \\
     0 & 0 & 1 & 0 \\
     0 & 0 & 0 & 1
\end{pmatrix}.
$$

Then we have $$\chi([h_1, h_2]) = z^{\frac{-d' a}{b'''}}  \lambda^{ -a e'} = \chi(h_4)^{N_1}= 1,$$  $$\chi([h_2, h_3]) =  w^{\frac{d' f''}{e'''}}  \lambda^{ b' f''} = \chi(h_5)^{N_2} = 1.$$ 
We also have $\chi(C)^{N_3} = \lambda^{N_3} = 1$, because $$\chi([h_1, h_3]) = \lambda^{-a e'' - b f ''} = 1,$$  $$\chi([h_3, h_4]) = \lambda^{b''' f ''} = 1$$ 
\begin{center}
and  
\end{center}
$$\chi([h_1, h_5]) = \lambda^{-a e'''} = 1.$$

In this case, the quotient $N_{G}(H) / H$ is clearly finite, and since $N_1, N_2$ and $N_3$ are chosen to be minimal natural numbers such that $z^{N_1} = w^{N_2} = \lambda^{N_3} = 1$, the action of $N_{G}(H) / H$ on characters is free. Hence, $S(H, \chi) = H$, and the corresponding representations are irreducible and finite-dimensional. Let us note that this case is the only one of finite-dimensional irreducible representations, all the others (cases $1-5$) refer to the infinite-dimensional ones.

\end{proof}

\medskip 

\begin{corollary}\label{corollary:fib6}

If $H \in \Sigma_{3, 2}$, then the fiber $Z_{3, 2; \, H}$ over $H$ is canonically bijective to:
$$\CC^{*} \times \CC^{*} \times \CC^{*} \times \mu_{N_1} \times \mu_{N_2} \times \mu_{N_3}.$$

\end{corollary} 

\begin{proof}
Since in this case $N_{G}(H) / H$ is finite, its action on a character $\chi$ does not change a conformal class of $Y_{3, 2; \, H}$.
\end{proof}

\medskip

\begin{lemma}\label{lemma:fib6}

If $H \in \Sigma_{3, 2}$, then there are only subgroups, which are $F$-equivalent to $H$ (finitely many subgroups).

Fibers $Z_{3, 2; \, H_1}$ and $Z_{3, 2;\, H_2} $ over equivalent subgroups $H_1$ and $H_2$ may be identified canonically.

\end{lemma}

\begin{proof}

If $H_1 \in \Sigma_{3, 2}$, then  $H_1^{*} = G$. By Proposition~\ref{lem:isopairs}[(ii)] we need to consider only equivalent irreducible weight pairs $(H_2, \chi_2)$ such that $H_1^{*} = H_{2}^{*} $ and $H_1~\neq~H_2$. 

Conditions for characters in this case are the following ones:$$\lambda^{f'' b'''} = \lambda^{a e'' + b f''} = \lambda^{a e'''} = 1,$$
 $$ 
w^{\frac{d' f''}{e'''}} \lambda^{ b' f''} = 1 \ \ \ \ \text{and} \ \ \  z^{\frac{-d' a}{b'''}}  \lambda^{ -a e'} = 1.$$ Let us denote $\mathrm{GCD}(f'' b''', a e'' + b f'', a e''')$ by $N_3$. 

We can possibly extract roots from generators $h_1$, $h_2$, $h_3$, $h_4$, and $h_5$. If there is $g_1 \in G \setminus H_1$ such that $g_1^{k_1} = h_1$, or  
$g_2 \in G \setminus H_1$ such that $g_2^{k_2} = h_2$, or $g_3 \in G \setminus H_1$ such that $g_3^{k_3} = h_3$, then by Proposition~\ref{lem:isopairs}[(ii)] for any divisor $m_1$ of $k_1$ or $m_2$ of $k_2$ or $m_3$ of $k_3$, there exists a finite set of exponents of other generators such that $S(H_2, \chi_2) = H_2$. Similarly, for any divisor $m_4$ of $b'''$ and any divisor $m_5$ of $e'''$, there also exists a finite set of corresponding exponents. 

Since $ a, d', f'', b''', e'''$ are fixed, once we fixed the exponents of generators of $H_1$, there are only parameters $b, b', b'' < |b'''|$ and $e, e', e'' < |e'''|$ which can be varied to produce equivalent irreducible weight pairs. Since we have $$ 
w^{\frac{d' f''}{e'''}} \lambda^{ b'  f''} = 1 \ \ \ \ \text{and} \ \ \  z^{\frac{-d' a}{b'''}}  \lambda^{ -a e'} = 1 \ \ \ \ \text{and} \ \ \  \lambda^{a e'' + b f''} = 1, $$
the expressions  $b'  f''$, $-a e'$ and $a e'' + b f''$ can not be changed modulo $N_3$, since $\lambda^{N_3} = 1$. So if we replace  $b, b', b'' < |b'''|$ and $e, e', e'' < |e'''|$ by $\tilde{b}, \tilde{b'}, \tilde{b''} < |b'''|$ and $\tilde{e}, \tilde{e'}, \tilde{e''} < |e'''|$ such that expressions above are unchanged modulo $N_3$, then there exists a corresponding character $\chi_2$ of a corresponding subgroup $H_2$ such that $\chi_1 \vert_{H_1 \cap H_2} =\chi_2 \vert_{H_1 \cap H_2} $. All such weight pairs are also $F$-equivalent to the weight pair $(H_1, \chi_1)$.

\end{proof}

\bigskip
\bigskip

\section{The main result}

Thus, we have finally obtained:

\medskip

\begin{theorem}\label{theorem:main}

There is a one-to-one correspondence between the following spaces:

$1.$ The union of the total spaces of the following bundles: $X_{1, 1} \rightarrow \Xi_{1,1}$, $X_{2, 0} \rightarrow \Xi_{2,0}$, $X_{2, 1} \rightarrow \Xi_{2,1}$, $X_{1, 2} \rightarrow \Xi_{1,2}$, $X_{2, 2} \rightarrow \Xi_{2,2}$, and $X_{3, 2} \rightarrow \Xi_{3,2}$. 

$2.$ A coarse moduli space of irreducible representations for the group of unipotent matrices of order $4$ with integer entries which have finite weight.

\medskip

A map from $X_{r_1, r_2} \rightarrow \Xi_{r_1,r_2}$ to the set of irreducible monomial representations is defined as follows: 
$$(H, \chi) \longmapsto \ind_{H}^{G}(\chi).$$\

\end{theorem}

\medskip

The fibers of these bundles are given in Corollaries~\ref{corollary:fib1},~\ref{corollary:fib2},~\ref{corollary:fib3},~\ref{corollary:fib4},~\ref{corollary:fib5},~\ref{corollary:fib6} and Lemmas~\ref{lemma:fib1}, ~\ref{lemma:fib2},~\ref{lemma:fib3},~\ref{lemma:fib4},~\ref{lemma:fib5},~\ref{lemma:fib6}. 
The definition of the bundle $X_{r_1, r_2} \rightarrow \Xi_{r_1,r_2}$ is given in~\ref{spaces}.

The fibers of the bundle $Z_{\mathrm{rk}_1, \mathrm{rk}_2} \rightarrow \Sigma_{\mathrm{rk}_1, \mathrm{rk}_2}$ are given in Table~\ref{table} below (see the corresponding definition in~\ref{z-equiv}).

\newpage

\begin{table}[]
\centering
\caption{Fibers of the bundle $Z_{\mathrm{rk}_1, \mathrm{rk}_2} \rightarrow \Sigma_{\mathrm{rk}_1, \mathrm{rk}_2}$}
\label{table}
\begin{tabular}{|l|l|l|l|l|l|l|l|l|}
\hline

\multicolumn{1}{|c|}{$\mathrm{rk}_1$} & $\mathrm{rk}_2$ & $\text{Subset}$ & $t$ & $r$ & $s$ & $z$ & $w$ & \multicolumn{1}{c|}{ $\lambda$}\\ \hline
 $1$ & $1$ & $\phi(S_1)$ & $ T_{ {z^d \lambda^{a'e + f'b + a'f'd}, \ a} }$ &  &  & $E_{\lambda^{2 a' f' }} \setminus \mu_{\infty}$ &  &  $\CC^* \setminus S^1$ \\ \hline
 $1$ & $1$ & $\phi(S_1)$ & $ T_{ {z^d \lambda^{a'e + f'b + a'f'd}, \ a} }$ &  &  & $P_{\lambda^{2 a' f' }} \setminus \mu_{\infty}$ &  &  $S^1 \setminus \mu_{\infty}$ \\ \hline
 $1$ & $1$ & $\phi(S_1)$ & $ E_{\lambda^{\mathrm{GCD}(a'e + f'b + a'f'd, a)}}$ &  &  & $\mu_{\infty} $ &  &  $\CC^* \setminus S^1$ \\ \hline
 $1$ & $1$ & $\phi(S_1)$ & $ P_{\lambda^{\mathrm{GCD}(a'e + f'b + a'f'd, a)}}$ &  &  & $\mu_{\infty} $ &  &  $S^1 \setminus \mu_{\infty}$ \\ \hline
 
  $1$ & $1$ & $\phi(S_1)$ & $ T_{ {z^d \lambda^{a'e + f'b + a'f'd}, \ a} }$ &  &  & $E_{\lambda^{2 a' f' }} \setminus \mu_{\infty}$ &  &  $\CC^* \setminus S^1$ \\ \hline
 $1$ & $1$ & $\phi(S_1)$ & $ T_{ {z^d \lambda^{a'e + f'b + a'f'd}, \ a} }$ &  &  & $P_{\lambda^{2 a' f' }} \setminus \mu_{\infty}$ &  &  $S^1 \setminus \mu_{\infty}$ \\ \hline
 $1$ & $1$ & $\phi(S_1)$ & $ E_{\lambda^{\mathrm{GCD}(a'e + f'b + a'f'd, a)}}$ &  &  & $\mu_{\infty} $ &  &  $\CC^* \setminus S^1$ \\ \hline
 $1$ & $1$ & $\phi(S_1)$ & $ P_{\lambda^{\mathrm{GCD}(a'e + f'b + a'f'd, a)}}$ &  &  & $\mu_{\infty} $ &  &  $S^1 \setminus \mu_{\infty}$ \\ \Xhline{2\arrayrulewidth}
 
 $1$ & $1$ & $\phi(S_2)$ & $ T_{ {z^d \lambda^{a'e + f'b + a'f'd}, \ a} }$ &  &  & $E_{\lambda^{2 a' f' }} \setminus \mu_{\infty}$ &  &  $\CC^* \setminus S^1$ \\ \hline
 $1$ & $1$ & $\phi(S_2)$ & $ T_{ {z^d \lambda^{a'e + f'b + a'f'd}, \ a} }$ &  &  & $P_{\lambda^{2 a' f' }} \setminus \mu_{\infty}$ &  &  $S^1 \setminus \mu_{\infty}$ \\ \hline
 $1$ & $1$ & $\phi(S_2)$ & $ E_{\lambda^{\mathrm{GCD}(a'e + f'b + a'f'd, a)}}$ &  &  & $\mu_{\infty} $ &  &  $\CC^* \setminus S^1$ \\ \hline
 $1$ & $1$ & $\phi(S_2)$ & $ P_{\lambda^{\mathrm{GCD}(a'e + f'b + a'f'd, a)}}$ &  &  & $\mu_{\infty} $ &  &  $S^1 \setminus \mu_{\infty}$ \\ \hline
 
  $1$ & $1$ & $\phi(S_2)$ & $ T_{ {z^d \lambda^{a'e + f'b + a'f'd}, \ a} }$ &  &  & $E_{\lambda^{2 a' f' }} \setminus \mu_{\infty}$ &  &  $\CC^* \setminus S^1$ \\ \hline
 $1$ & $1$ & $\phi(S_2)$ & $ T_{ {z^d \lambda^{a'e + f'b + a'f'd}, \ a} }$ &  &  & $P_{\lambda^{2 a' f' }} \setminus \mu_{\infty}$ &  &  $S^1 \setminus \mu_{\infty}$ \\ \hline
 $1$ & $1$ & $\phi(S_2)$ & $ E_{\lambda^{\mathrm{GCD}(a'e + f'b + a'f'd, a)}}$ &  &  & $\mu_{\infty} $ &  &  $\CC^* \setminus S^1$ \\ \hline
 $1$ & $1$ & $\phi(S_2)$ & $ P_{\lambda^{\mathrm{GCD}(a'e + f'b + a'f'd, a)}}$ &  &  & $\mu_{\infty} $ &  &  $S^1 \setminus \mu_{\infty}$ \\ \Xhline{2\arrayrulewidth}

 $1$ & $1$ & $\phi(S_3)$ & $ T_{z^d \lambda^{a'e}, \lambda^{a}}$ &  &  & $\CC^* \setminus S^1$ &  &  $\CC^* \setminus S^1$ \\ \hline
 $1$ & $1$ & $\phi(S_3)$ & $ T_{z^d \lambda^{a'e}, \lambda^{a}}$ &  &  & $S^1 \setminus \mu_{\infty}$ &  &  $\CC^* \setminus S^1$ \\ \hline
 $1$ & $1$ & $\phi(S_3)$ & $ T_{z^d \lambda^{a'e}, \lambda^{a}} $ &  &  & $\CC^* \setminus \mu_{\infty}$ &  &  $S^1 \setminus \mu_{\infty}$ \\ \hline
 $1$ & $1$ & $\phi(S_3)$ & $  E_{\lambda^{\mathrm{GCD}(a'e, a)}} $ &  &  & $\mu_{\infty}  $ &  &  $\CC^* \setminus S^1$ \\ \hline
 $1$ & $1$ & $\phi(S_3)$ & $ P_{\lambda^{\mathrm{GCD}(a'e, a)}} $ &  &  & $\mu_{\infty} $ &  &  $S^1 \setminus \mu_{\infty}$ \\ \Xhline{2\arrayrulewidth}
 
 $1$ & $1$ & $\phi(S_4)$ & $ E_{\lambda^{\mathrm{GCD}(a'e + f'b, a)}}$ &  &  & $E_{\lambda^{2 a' f'}}$ &  &  $\CC^* \setminus S^1$ \\ \hline
 $1$ & $1$ & $\phi(S_4)$ & $ P_{\lambda^{\mathrm{GCD}(a'e + f'b, a)}}$ &  &  & $P_{\lambda^{2 a' f'}}$ &  &  $S^1 \setminus \mu_{\infty}$ \\ \Xhline{2\arrayrulewidth}
 
 $1$ & $1$ & $\phi(N_1)$ & $ T_{z^{a}, \lambda^{a}}$ &  &  & $E_{\lambda} \setminus \mu_{\infty}$ &  &  $\CC^* \setminus S^1$ \\ \hline
 $1$ & $1$ & $\phi(N_1)$ & $  T_{z^{a}, \lambda^{a}}$ &  &  & $P_{\lambda} \setminus \mu_{\infty}$ &  &  $S^1 \setminus \mu_{\infty}$\\ \hline
 $1$ & $1$ & $\phi(N_1)$ & $ E_{\lambda}$ &  &  & $\mu_{\infty}$ &  &  $\CC^* \setminus S^1$ \\ \hline
 $1$ & $1$ & $\phi(N_1)$ & $ P_{\lambda}$ &  &  & $\mu_{\infty}$ &  &  $S^1 \setminus \mu_{\infty}$  \\ \Xhline{2\arrayrulewidth}
 
 $1$ & $1$ & $\phi(N_2)$ & $ T_{z^{f}, \lambda^{f}} $ &  &  & $E_{\lambda} \setminus \mu_{\infty}$ &  &  $\CC^* \setminus S^1$ \\ \hline
 $1$ & $1$ & $\phi(N_2)$ & $ T_{z^{f}, \lambda^{f}} $ &  &  & $P_{\lambda} \setminus \mu_{\infty}$ &  &  $S^1 \setminus \mu_{\infty}$\\ \hline
 $1$ & $1$ & $\phi(N_2)$ & $ E_{\lambda}$ &  &  & $\mu_{\infty}$ &  &  $\CC^* \setminus S^1$ \\ \hline
 $1$ & $1$ & $\phi(N_2)$ & $ P_{\lambda}$ &  &  & $\mu_{\infty}$ &  &  $S^1 \setminus \mu_{\infty}$  \\ \Xhline{2\arrayrulewidth}

 $2$ & $0$ & $\Sigma_{2, 0}$ & $ E_{\lambda^{f'}} $ &  & $E_{\lambda^a} $ &  &  &  $\CC^* \setminus S^1$ \\ \hline
 $2$ & $0$ & $\Sigma_{2, 0}$ & $ P_{\lambda^{f'}} $ &  & $P_{\lambda^a} $ &  &  &  $S^1 \setminus \mu_{\infty}$   \\ \Xhline{2\arrayrulewidth}

 $2$ & $1$ & $\Sigma_{2, 1}$ & $ E_{\lambda^a} $ & $ \CC^{*}$ & & $ \mu_{a d'} $ &  &  $\CC^* \setminus S^1$ \\ \hline
 $2$ & $1$ & $\Sigma_{2, 1}$ & $ P_{\lambda^a}$ & $ \CC^{*}$ & & $ \mu_{a d'} $  &  &  $S^1 \setminus \mu_{\infty}$   \\ \Xhline{2\arrayrulewidth}
 
  $1$ & $2$ & $\Sigma_{1, 2} \setminus \phi(A)$ & $T_{z^{\frac{ad'}{b'}} w^{\frac{fd'}{e'}}, \, w^{\frac{f'd}{e'}} \lambda^{b f' }} $ &  &  & $\CC^{*} \setminus \mu_{\infty}$ & $\CC^{*} \setminus \mu_{\infty}$ & $\mu_{N}$ \\ \Xhline{2\arrayrulewidth}

\end{tabular}
\end{table}
\newpage

\begin{table}[]
\centering
\begin{tabular}{|l|l|l|l|l|l|l|l|l|}
\hline

\multicolumn{1}{|c|}{$1$} &  $2$ & $\phi(A)$ & $T_{z, w} $ &  &  & $E_{\lambda}  \setminus \mu_{\infty}$ & $E_{\lambda}  \setminus \mu_{\infty}$ & \multicolumn{1}{c|}{ $ \CC^{*} \setminus S^1$}\\ \hline
 $1$ & $2$ & $\phi(A)$ & $T_{z, w} $ &  &  & $P_{\lambda}  \setminus \mu_{\infty} $ & $P_{\lambda}  \setminus \mu_{\infty} $ & $P_{\lambda}  \setminus \mu_{\infty} $ \\ \hline
 $1$ & $2$ & $\phi(A)$ & $ E_{z} $ &  &  & $E_{\lambda} \setminus \mu_{\infty} $ & $E_{\lambda} \setminus \mu_{\infty} $ & $\CC^{*} \setminus S^1$ \\ \hline
 $1$ & $2$ & $\phi(A)$ & $ E_{z} $ &  &  & $P_{\lambda} \setminus \mu_{\infty}$ & $P_{\lambda} \setminus \mu_{\infty}$ & $S^1 \setminus \mu_{\infty}$ \\ \hline
 
 $1$ & $2$ & $\phi(A)$ & $E_{w} $ &  &  & $E_{\lambda} \cap  \mu_{\infty}$ & $E_{\lambda} \cap  \mu_{\infty}$ & $ \CC^{*} \setminus  S^1$ \\ \hline
 $1$ & $2$ & $\phi(A)$ & $E_{w} $ &  &  & $P_{\lambda} \cap  \mu_{\infty}$ & $P_{\lambda} \cap  \mu_{\infty}$ & $S^1 \setminus \mu_{\infty}$ \\ \hline
 
 $1$ & $2$ & $\phi(A)$ & $ \CC^{*} $ &  &  & $E_{\lambda} \cap  \mu_{\infty}$ & $E_{\lambda} \cap  \mu_{\infty}$ & $ \CC^{*} \setminus  S^1 $ \\ \hline
 $1$ & $2$ & $\phi(A)$ & $ \CC^{*} $ &  &  & $P_{\lambda} \cap  \mu_{\infty} $ & $P_{\lambda} \cap  \mu_{\infty} $ & $S^1 \setminus \mu_{\infty} $ \\ \hline
 
 $1$ & $2$ & $\phi(A)$ & $T_{z,w} $ &  &  & $\CC^{*} \setminus \mu_{\infty}$ & $\CC^{*} \setminus \mu_{\infty}$ & $ \mu_{\infty} $ \\ \hline
 $1$ & $2$ & $\phi(A)$ & $T_{z,w} $ &  &  & $\CC^{*} \setminus \mu_{\infty}$ & $\CC^{*} \setminus \mu_{\infty}$ & $ \mu_{\infty} $ \\ \Xhline{2\arrayrulewidth}

  $2$ & $2$ & $\phi(S_1)$ & $\CC^{*}$ &  & $ E_{w^{\frac{d'\tilde{f}}{e''}} } $ & $C_{z, w}$ & $C_{z, w}$ & $ \mu_{N}$ \\ \hline
  
  $2$ & $2$ & $\phi(S_1)$ & $\CC^{*}$ &  & $ P_{w^{\frac{d'\tilde{f}}{e''}} } $  & $C_{z, w, sing}$ & $C_{z, w, sing}$ & $\mu_{N}$ \\ \hline

  $2$ & $2$ & $\phi(S_1)$ & $\CC^{*}$ &  & $ E_{w^{\frac{d'\tilde{f}}{e''}} } $ & $C_{z, w}$ & $C_{z, w}$ & $ \mu_{N}$ \\ \hline
  
 $2$ & $2$ & $\phi(S_1)$ & $\CC^{*}$ &  & $ P_{w^{\frac{d'\tilde{f}}{e''}} }$  & $C_{z, w, sing}$ & $C_{z, w, sing}$ & $\mu_{N}$ \\ \hline
 
 $2$ & $2$ & $\phi(S_2)$ & $ \CC^{*} $ &  & $E_{w^{\frac{d'\tilde{f}}{e''}} }$ & $\mu_{N_2}$ & $\CC^* \setminus S^1 $ & $\mu_{a e''}$ \\ \hline
 
$2$ & $2$ & $\phi(S_2)$ & $ \CC^{*} $ &  & $P_{w^{\frac{d'\tilde{f}}{e''}} } $ & $\mu_{N_2}$ & $S^1 \setminus \mu_{\infty} $ & $\mu_{a e''}$ \\ \hline
 
 $2$ & $2$ & $\phi(S_3)$ & $E_{z^{\frac{d' a}{b''}}} $ &  & $ \CC^{*} $ & $ \CC^* \setminus S^1$ & $ \mu_{N_3} $ & $\mu_{f' b''}$ \\ \hline
 
$2$ & $2$ & $\phi(S_3)$ & $P_{z^{\frac{d' a}{b''}}} $ &  & $ \CC^{*} $ & $S^1 \setminus \mu_{\infty} $ & $ \mu_{N_3} $ & $\mu_{f' b''}$ \\ \hline

  $2$ & $2$ & $\phi(S_4)$ & $E_{z^{\frac{d' a}{b''}}} $ &  & $E_{w^{\frac{d'\tilde{f}}{e''}} } $ & $\CC^{*} \setminus S^1$ & $\CC^{*} \setminus S^1$ & $ \mu_{N}$ \\ \hline
  
  $2$ & $2$ & $\phi(S_4)$ & $P_{z^{\frac{d' a}{b''}}} $ &  & $E_{w^{\frac{d'\tilde{f}}{e''}} }$  & $S^1 \setminus \mu_{\infty}$ & $\CC^{*} \setminus S^1$ & $\mu_{N}$ \\ \hline
  $2$ & $2$ & $\phi(S_4)$ & $E_{z^{\frac{d' a}{b''}}} $ &  & $E_{w^{\frac{d'\tilde{f}}{e''}} } $ & $\CC^{*} \setminus S^1$ & $\CC^{*} \setminus S^1$ & $ \mu_{N}$ \\ \hline
  $2$ & $2$ & $\phi(S_4)$ & $P_{z^{\frac{d' a}{b''}}} $ &  & $P_{w^{\frac{d'\tilde{f}}{e''}} }$  & $S^1 \setminus \mu_{\infty}$ & $S^1 \setminus \mu_{\infty}$ & $\mu_{N}$ \\ \hline
  
   $2$ & $2$ & $\phi(S_4)$ & $\CC^{*}$ &  & $E_{w^{\frac{d'\tilde{f}}{e''}} } $ & $\mu_{\infty}$ & $\CC^{*} \setminus S^1$ & $ \mu_{N}$ \\ \hline
  
  $2$ & $2$ & $\phi(S_4)$ & $E_{z^{\frac{d' a}{b''}}} $ &  & $\CC^{*} $  & $\CC^{*} \setminus S^1$ & $\mu_{\infty}$ & $\mu_{N}$ \\ \hline
  
  $2$ & $2$ & $\phi(S_4)$ & $\CC^{*}$ &  & $P_{w^{\frac{d'\tilde{f}}{e''}} } $ & $\mu_{\infty}$ & $S^1 \setminus \mu_{\infty}$ & $ \mu_{N}$ \\ \hline
  
  $2$ & $2$ & $\phi(S_4)$ & $P_{z^{\frac{d' a}{b''}}} $ &  & $\CC^{*}$  & $S^1 \setminus \mu_{\infty}$ & $\mu_{\infty}$ & $\mu_{N}$ \\ \Xhline{2\arrayrulewidth}

  $3$ & $2$ & $\Sigma_{3, 2}$ & $\CC^{*} $ & $\CC^{*}$  & $\CC^{*}$  & $\mu_{N_1}$ & $\mu_{N_2}$ & $\mu_{N_3}$ \\ \Xhline{2\arrayrulewidth}

\end{tabular}
\end{table}

\medskip

{\bf Acknoledgements.} The author is grateful to S.\,Gorchinskiy for useful discussions and suggestions. The author was supported by the National Centre of Competence in Research ``SwissMAP The Mathematics of Physics" of the Swiss National Science Foundation. The author is also grateful for hospitality and excellent working conditions to the Max Planck Institute for Mathematics, where a part of the work was done.

\end{document}